\documentclass[a4paper, reqno]{amsart}
\usepackage{amssymb,amscd,amsfonts,amsbsy,nicefrac,mathtools}
\usepackage[utf8]{inputenc}
\usepackage{amsmath,amsthm,amssymb,amsfonts,enumitem}
\usepackage{mathrsfs}
\usepackage[nobysame]{amsrefs}
\usepackage{hyperref}
\usepackage[pdftex]{graphicx,color}
\usepackage{graphicx}
\usepackage{marvosym}
\usepackage{scalerel}
\usepackage{bbm}
\usepackage{bm}
\usepackage{setspace}
\usepackage{tikz}
\usepackage{soul}

\usepackage[T1]{fontenc}

\usepackage{fancyhdr} 
\usepackage[a4paper,tmargin=3cm,bmargin=3cm,lmargin=3cm,rmargin=3cm]{geometry}

\newtheorem{thm}{Theorem}[section]
\newtheorem{rem}[thm]{Remark}
\newtheorem{lem}[thm]{Lemma}

\newtheorem{cor}[thm]{Corollary}
\newtheorem{prop}[thm]{Proposition}
\newtheorem{defn}[thm]{Definition}

\newcommand{\R}{\mathbb{R}}
\newcommand{\bbR}{\mathbb{R}}
\newcommand{\C}{\mathbb C}
\newcommand{\bbC}{\mathbb{C}}

\newcommand{\bbZ}{\mathbb{Z}}
\newcommand{\N}{\mathbb{N}}

\newcommand{\supp}{\operatorname{supp}}

\renewcommand{\d}{\partial}

\DeclareMathOperator{\lcm}{lcm}

\DeclareFontFamily{U}{mathx}{}
\DeclareFontShape{U}{mathx}{m}{n}{<-> mathx10}{}
\DeclareSymbolFont{mathx}{U}{mathx}{m}{n}
\DeclareMathAccent{\widehat}{0}{mathx}{"70}
\DeclareMathAccent{\widecheck}{0}{mathx}{"71}

\newcommand{\abs}[1]{\left|#1\right|}
\newcommand{\set}[1]{\left\{#1\right\}}

\newcommand{\m}[1]{\begin{equation*}
\begin{split}
#1
\end{split}
\end{equation*}}
\newcommand{\nm}[2]{\begin{equation}\label{#1}
\begin{split}
#2
\end{split}
\end{equation}}

\usepackage[normalem]{ulem}

\begin{document}

\title{Determinantal Formulas for Rational Perturbations of
Multiple Orthogonality Measures}
\date{\today}
\author{Rostyslav Kozhan$^{1}$}
\author{Marcus Vaktnäs$^{2}$}
\address{$^{1}$Department of Mathematics, Uppsala University, S-751 06 Uppsala, Sweden, rostyslav.kozhan@math.uu.se}
\address{$^{1}$Department of Mathematics, Uppsala University, S-751 06 Uppsala, Sweden, marcus.vaktnas@math.uu.se}
\begin{abstract}
    Given multiple orthogonal polynomials on the real line with respect to a system $\bm{\mu} = (\mu_1,\ldots,\mu_r)$, we investigate multiple orthogonal polynomials associated with any 
    rational perturbation of the form
    $$
    \widetilde{\bm\mu}=\Big(\frac{\Phi_1}{\Psi_{1}} \mu_1,\dots,\frac{\Phi_r}{\Psi_r}\mu_r\Big),
    $$
    for any polynomials $\Phi_1,\dots,\Phi_r$ and $\Psi_1,\dots,\Psi_r$. 
    We derive the analogues of Uvarov's determinantal formula for the multiple orthogonal polynomials  of type I and type II for $\widetilde{\bm\mu}$ and establish necessary and sufficient condition  for normality of the indices.

    The result allows the polynomials $\{\Phi_j,\Psi_j\}_{j=1}^r$ to be arbitrary and permits the addition of finitely many point masses to each of the measures $\mu_j$. Moreover, the measures $\mu_j$ may be taken as quasi-definite linear functionals, which is of interest even in the case $r=1$. 

\end{abstract}
\maketitle

\section{Introduction}

For a positive Borel measure $\mu$ on $\R$ with infinite support and finite moments, we have orthogonal polynomials $(P_n)_{n = 0}^{\infty}$ defined by 
the orthogonality condition
\nm{eq:orthogonality conditions one measure}{\int P_n(x)x^p d\mu(x) = 0, \qquad p = 0,\dots,n-1,}
along with the degree condition $\deg{P_n} \leq n$.
Each non-zero $P_n$ necessarily has degree $n$ and is unique up to multiplication by a constant. 
If we normalize 
 the polynomials to be monic (having leading coefficient $1$), then these polynomials are uniquely determined by 
\eqref{eq:orthogonality conditions one measure}. 

Now consider the condition
\nm{orthogonality conditions ct one measure}{\int \widehat{P}_n(x)x^p\Phi(x)d\mu(x) = 0, \qquad p = 0,\dots,n-1,}
along with the degree condition $\deg{\widehat{P}_n} \leq n$, for a polynomial $\Phi(x) = \prod_{k = 1}^N(x-z_k)$. If $\Phi$ is positive on $\supp(\mu)$ then \eqref{orthogonality conditions ct one measure} is just \eqref{eq:orthogonality conditions one measure} for the positive measure $\widehat{\mu}$ defined by
\nm{eq:ct one measure}{\int f(x) d\widehat{\mu}(x) = \int f(x) \Phi(x)d\mu(x).}
Note that this measure is also Borel with infinite support and finite moments. Assuming $\Phi$ has only simple zeros, Christoffel's Theorem \cite{Christoffel} states that the monic orthogonal polynomials $\widehat{P}_n$ with respect to $\widehat{\mu}$ are given by
\nm{eq:ct thm one measure}{\widehat{P}_n(x) = {\Phi(x)}^{-1}D_n^{-1}\det\begin{pmatrix}
P_{n+N}(x) & P_{n+N-1}(x) & \cdots & P_n(x) \\
P_{n+N}(z_1) & P_{n+N-1}(z_1) & \cdots & P_n(z_1) \\
\vdots & \vdots & \ddots & \vdots \\
P_{n+N}(z_N) & P_{n+N-1}(z_N) & \cdots & P_n(z_N) \\ 
\end{pmatrix},}
where $D_n$ is the normalizing constant
\nm{eq:dn ct one measure}{D_n = \det\begin{pmatrix}
P_{n+N-1}(z_1) & \cdots & P_n(z_1) \\
\vdots & \ddots & \vdots \\
P_{n+N-1}(z_N) & \cdots & P_n(z_N) \end{pmatrix}.}

We refer to $\widehat{\mu}$ as a Christoffel transform of $\mu$. 
If $\deg{\Phi} = 1$ we say that $\widehat{\mu}$ is a one-step Christoffel transform, and then \eqref{eq:ct thm one measure} simplifies to
\nm{eq:one step ct thm one measure}{\widehat{P}_n(x) = (x-z_1)^{-1} \left(P_{n+1}(x) - \frac{P_{n+1}(z_1)}{P_n(z_1)} P_n(x)\right).}
It is well-known that the zeros of $P_n$ lie in the interior of every interval containing $\supp(\mu)$, so the denominator in \eqref{eq:one step ct thm one measure} will always be non-zero, assuming $z_1$ is outside the convex hull of $\supp(\mu)$. More generally, one can show that \eqref{eq:dn ct one measure} is also non-zero, at least when $\Phi$ does not change sign on $\supp(\mu)$. 

Consider the more general rational transformations $\widetilde{\mu}$ defined by
\nm{eq:rp one measure}{\int f(x)d\widetilde{\mu}(x) = \int f(x)\frac{\Phi(x)}{\Psi(x)}d\mu(x),}
for polynomials $\Phi(x) = \prod_{k = 1}^N(x-z_k)$ and $\Psi(x) = \prod_{k = 1}^M(x-w_k)$ such that $\Phi/\Psi$ is positive and integrable on $\supp(\mu)$. Uvarov \cite{Uvarov} showed that the monic orthogonal polynomials $\widetilde{P}_n$ with respect to the measure $\widetilde{\mu}$ are given by
\nm{eq:rp thm one measure}{\widetilde{P}_n(x) = {\Phi(x)}^{-1}D_n^{-1}\det\begin{pmatrix}
P_{n+N}(x) & P_{n+N-1}(x) & \cdots & P_{n-M}(x) \\
P_{n+N}(z_1) & P_{n+N-1}(z_1) & \cdots & P_{n-M}(z_1) \\
\vdots & \vdots & \ddots & \vdots \\
P_{n+N}(z_N) & P_{n+N-1}(z_N) & \cdots & P_{n-M}(z_N) \\ 
Q_{n+N}(w_1) & Q_{n+N-1}(w_1) & \cdots & Q_{n-M}(w_1) \\
\vdots & \vdots & \ddots & \vdots \\
Q_{n+N}(w_M) & Q_{n+N-1}(w_M) & \cdots & Q_{n-M}(w_M)
\end{pmatrix}, \qquad n \geq M.}
Here $Q_n(z)$ are the so-called functions of the second kind defined by
\nm{eq:q rp one measure}{Q_n(z) = \int \frac{P_n(x)}{x-z} d\mu(x),}
and $D_{n}$ is the normalizing constant
\nm{eq:dn rp one measure}{D_n = \det{\begin{pmatrix}
P_{n+N-1}(z_1) & \cdots & P_{n-M}(z_1) \\
\vdots & \ddots & \vdots \\
P_{n+N-1}(z_N) & \cdots & P_{n-M}(z_N) \\ 
Q_{n+N-1}(w_1) & \cdots & Q_{n-M}(w_1) \\
\vdots & \ddots & \vdots \\
Q_{n+N-1}(w_M) & \cdots & Q_{n-M}(w_M)
\end{pmatrix}}.}
This determinant is always non-zero under the assumption that $\Phi/\Psi$ is positive on $\supp(\mu)$. 




The current paper {provides}  a generalization of the formula \eqref{eq:rp thm one measure} to the multiple orthogonality setting. Multiple orthogonal polynomials are polynomials that satisfy {simultaneous} orthogonality conditions with respect to multi-indices $(n_1,\dots,n_r)$ and a system of measures $(\mu_1,\dots,\mu_r)$, where the case $r=1$ reduces to the standard orthogonal polynomials. {For an introduction to multiple orthogonal polynomials with applications, see \cite{Aptekarev,Ismail,Applications}.} The rational transformations of $(\mu_1,\dots,\mu_r)$ considered in this paper are the systems $(\widetilde{\mu}_1,\dots,\widetilde{\mu}_r)$ of the form
\nm{eq:rp several measures}{\int f(x) d\widetilde{\mu}_j(x) = \int f(x)\frac{\Phi_j(x)}{\Psi_j(x)}d\mu_j(x) + \sum_{k = 1}^{M_j}c_{j,k}f({w_{j,k}}), \qquad j = 1,\dots,r.}

Such a rational perturbation of $\mu$ is not necessarily a positive measure, 
so we {choose to} work in the more general setting of orthogonality with respect to linear moment functionals. 
We prove the determinantal formulas, both for type I and for type II polynomials, and give a necessary and sufficient condition for their existence and uniqueness (normality). This condition is 
a simple determinantal condition $D_{\bm{n}} \ne 0$, where $D_{\bm{n}}$ is an appropriately chosen analogue of~\eqref{eq:dn rp one measure}. 


For the case $r=1$, the transformation~\eqref{eq:rp several measures} with $\Phi(x)=1$ is called the Geronimus transform~\cite{Geronimus} and ~\eqref{eq:rp several measures} with $\Phi(x)=\Psi(x)$ is called the Uvarov transform~\cite{Uvarov}.
Christoffel, Geronimus, and Uvarov transforms, as well as the closely connected topic of the Darboux transformation, are important in pure and applied mathematics. An inexhaustive list of papers on these topics is~\cite{BaiDer,BueMar,Chihara,Chi85,Christoffel,DerGarMar,GruHei,MarMar,Mar90,Mar91,GesTes,Peh92,SpiZhe,Uvarov,Zhe97} and~\cite{Bueno and Dopico,BDT,Galant,Gal92,Gautschi,Gautschi book,Golub and Kautsky,GolKau83,KauGol}.  

For the case of multiple orthogonality, results in this direction include: ~\cite{ADMVA} for one-step Geronimus and Christoffel transforms of type II polynomials; ~\cite{BFM23} for one-step Christoffel transforms of type I and type II polynomials;  \cite{ctpaper} treating multi-step Christoffel transform; 
and ~\cite{ManasRojas} addressing the multi-step Christoffel for mixed multiple orthogonality for indices along the step-line.
One-step multiple Christoffel transform is a tool that also appeared naturally in~\cite{AKVI,BDFLM,BFM22,KVInterlacing}.



{Upon completion of this manuscript, the independent work \cite{ManasRojas2} appeared addressing the general Geronimus transformation for multiple orthogonal systems. Their method of proof and formulas differ substantially from ours, and their approach is restricted to step-line indices and without considering combined Christoffel and Geronimus perturbations.
Their results treat the more general {\it mixed} multiple orthogonality setting, and their perturbations have the form of a polynomial matrix multiplication/division. This allows them to treat linear combinations of $\widetilde{\mu}_j$ in~\eqref{eq:rp several measures}, a very interesting setting that is currently beyond our technical reach. 
}



The organization of the paper is as follows. We remind the reader of the basics of orthogonal and multiple orthogonal polynomials with respect to linear functionals in Sections \ref{moment functionals}--\ref{moprl}. In Section \ref{type I formula} we prove a generalization of Uvarov's formula for type I polynomials, and in Section \ref{type II formula} we do the same for type II polynomials. Some understanding of the duality between the type I and type II setting reveals the proofs to be parallel. The proof of the type I formula requires $\Psi_1 = \cdots = \Psi_r$, while the proof of the type II formula requires $\Phi_1 = \cdots = \Phi_r$. In Section \ref{general case} we then describe how these formulas can be generalized to the case of any rational transformation. 

Through Uvarov's formula \eqref{eq:rp thm one measure}, $\Phi\widetilde{P}_n$ is expressed as a linear combination of $\set{P_{n+N-j}}_{j = 0}^{N+M}$. When we generalize it to multiple orthogonal polynomials, there are several possible sequences of multi-indices we can consider. In Section \ref{linear indepence of regular sequences} we define the most natural sequences (paths and frames), as well as a more general class of sequences containing the most natural ones, for which the determinantal formula holds. We remark that for type I polynomials we consider sequences moving upwards (see Definition \ref{def:forward sequences}), and for type II polynomials we consider sequences moving downwards (see Definition \ref{def:backward sequences}).  

In Section \ref{examples} we specialize our general formula to a number of natural special cases that might be of interest. Here the duality of the type I and type II setting is further illustrated through the comparison of the formulas for Christoffel transforms 
and Geronimus transforms,
as well as the comparison of {the systems} $(\Phi\mu_1,\Phi\mu_2\dots,\Phi\mu_r)$ and $(\Phi\mu_1,\mu_2\dots,\mu_r)$. We also note that Uvarov perturbations (adding point masses) arise as special cases of our formulas (the case $\Phi = \Psi$). The reader may find the generality of Section \ref{the general formula} easier to parse by reading it alongside Section \ref{examples} to compare. 

Another interesting observation is that our determinantal formulas for the systems of the form $(\Phi\mu_1,\mu_2,\dots,\mu_r)$ requires a generalization of \eqref{eq:rp thm one measure}, rather than just a generalization of \eqref{eq:ct thm one measure}. This is because our type II formula require $\Phi_1 = \dots = \Phi_r$, so we view the system $(\Phi\mu_1,\mu_2,\dots,\mu_r)$ as $(\frac{\Phi}{1}\mu_1,\frac{\Phi}{\Phi}\mu_2\dots,\frac{\Phi}{\Phi}\mu_r)$, and then generalizations of the $Q$'s of \eqref{eq:q rp one measure} enter into the determinantal formula (this is Section \ref{ex:partial ct type II}). This somewhat surprising observation provided a motivation for studying the more general rational perturbations rather than just Christoffel perturbations. Note however that for the type I formula we can view $(\Phi\mu_1,\mu_2\dots,\mu_r)$ as $(\frac{\Phi}{1}\mu_1,\frac{1}{1}\mu_2\dots,\frac{1}{1}\mu_r)$, and indeed no $Q$'s are needed there (Section \ref{ex:partial ct type I}).

Finally, in Section~\ref{ss:Laguerre} as a case study, we apply Theorems~\ref{thm:type I thm} and \ref{thm:type II thm} to a specific perturbation~\eqref{eq:LagPert1}--\eqref{eq:LagPert2} of the multiple Laguerre polynomials of the second kind.


It should be noted that the determinantal formulas for the one-step Christoffel transform in our paper~\cite{ctpaper} naturally lead to interlacing results between zeros of the multiple orthogonal polynomials for the initial and for the transformed systems. In particular this results in interlacing for the polynomials from the same classical multiple orthogonal sytems but with different parameters. It is straightforward to obtain analogous interlacing results from the determinantal formulas for the one-step Christoffel and one-step Geronimus transforms treated in this paper. This will be addressed in a forthcoming publication.

\subsubsection*{Acknowledgements}
The authors are grateful to the anonymous referees whose comments and suggestions helped to improve the paper. 

\section{Preliminaries}

\subsection{Rational Perturbations of Moment Functionals}\label{moment functionals}\hfill\\

The transformed orthogonality measure in~\eqref{orthogonality conditions ct one measure} is not necessarily a positive measure. To deal with this issue we view the measure $\mu$ as a linear functional 
\nm{eq:linear functional}{\mu[P(x)] = \int P(x) d\mu(x).}
\begin{defn}
    Let $(\nu_j)_{j=0}^\infty$ be a sequence of complex numbers, which we refer to as the moments. Let $\mu$ be the linear functional on the space of polynomials with complex coefficients defined by $\mu[x^j] = \nu_j$ and extended by linearity. 
    The orthogonal polynomials $(P_n)_{n=0}^{\infty}$ with respect 
    $\mu$ are defined by the orthogonality condition
\nm{eq:orthogonality conditions one functional}{\mu\left[P_n(x)x^p\right] = 0, \qquad p = 0,\dots,n-1,}
and the degree condition
\nm{eq:degree condition one functional}{\deg P_n \leq n.} 
\end{defn}

For positive definite $\mu$ this corresponds to orthogonality with respect to infinitely supported measures. If $\mu$ is not given by a positive measure we may have non-zero solutions to \eqref{eq:orthogonality conditions one functional}-\eqref{eq:degree condition one functional} with $\deg{P_n} < n$, and the solution may not be unique up to multiplication by a constant. These cases need to be excluded, so we use the following definition. 
\begin{defn}
    We say that $n \in \N$ is normal for $\mu$ if $\deg P_n = n$ for every non-zero solution of \eqref{eq:orthogonality conditions one functional}-\eqref{eq:degree condition one functional} at $n$. We say that $\mu$ is quasi-definite if every $n \in \N$ is normal.
\end{defn}
It is easy to see that normality of $n$ is equivalent to the existence and uniqueness of a {degree $n$} monic solution to~\eqref{eq:orthogonality conditions one functional}-\eqref{eq:degree condition one functional}. If we look for non-zero solutions of degree $<n$, i.e., of the form $P_n(x) = \kappa_{n-1}x^{n-1} + \cdots + \kappa_0$, then \eqref{eq:orthogonality conditions one functional} turns into a linear system of equations for the coefficients $\kappa_{n-1},\dots,\kappa_0$. The coefficient matrix of this system is the Hankel matrix given by
\begin{equation}\label{eq:matrix}
{H_n = \begin{pmatrix}
\nu_0 & \nu_1 & \cdots & \nu_{n-1} \\
\nu_1 & \nu_2 & \cdots & \nu_{n} \\
\vdots & \vdots & \ddots & \vdots \\
\nu_{n-1} & \nu_n & \cdots & \nu_{2n-2} \\
\end{pmatrix}.}
\end{equation}
Hence $n$ is normal for $\mu$ if and only if $\det{H_n} \neq 0$ (except in the trivial case $n = 0$, where we always take $n$ to be normal). For more about orthogonal polynomials with respect to quasi-definite functionals, see \cite{Chihara}. 

\begin{defn}
A Christoffel transform of the functional $\mu$ is a functional $\widehat{\mu}$ given by 
\nm{eq:ct definition one functional}{\widehat{\mu}[P(x)] = \mu[\Phi(x)P(x)].}
for some polynomial $\Phi(x) = \prod_{k = 1}^{N}(x-z_k)$. If $\deg{\Phi} = 1$ then we say that $\widehat{\mu}$ is a one-step Christoffel transform.
\end{defn}

We employ the notation $\widehat{\mu} = \Phi\mu$. As long as $\mu$ and $\widehat{\mu}$ are quasi-definite we still have \eqref{eq:ct thm one measure} and \eqref{eq:one step ct thm one measure}. If $\Phi(x) = x - z_1$ and $\mu$ is quasi-definite then it is well known that $\widehat{\mu}$ is quasi-definite if and only if $P_n(z_1) \neq 0$ for each $n \in \N$. More generally, for any $\Phi$ with simple zeros, if $\mu$ is quasi-definite then $\Phi\mu$ is quasi-definite if and only if $D_n \neq 0$ for each $n \in \N$ where $D_n$ is given by \eqref{eq:dn ct one measure}, {see, e.g.,~\cite{ctpaper}}.

\begin{defn}\label{defn:geronimus}
    A Geronimus transform of the functional $\mu$ is a functional $\widecheck{\mu}$ satisfying 
\nm{eq:gt definition one functional}{\widecheck{\mu}[\Psi(x)P(x)] = \mu[P(x)]}
for a polynomial $\Psi(x) = \prod_{k = 1}^M(x-w_k)$. 
\end{defn}
Note that this only uniquely determines $\widecheck{\mu}$ on a subspace of codimension $M$ (the subspace of polynomials divisible by $\Psi$). Indeed, the first $M$ moments of $\widecheck{\mu}$ can be chosen arbitrarily, and together with~\eqref{eq:gt definition one functional} the other moments are then uniquely determined. In other words, if we express the polynomial $P$ as
\begin{equation}
    P(x) = \sum_{j = 0}^{M-1}\kappa_jx^j + \sum_{j = n}^{\infty}\kappa_j\Psi(x)x^j,
\end{equation}
then $\widecheck{\mu}$ has to be given by
\begin{equation}
    \widecheck{\mu}[P(x)] = \sum_{j = 0}^{M-1}\kappa_j\widecheck{\mu}[x^j] + \sum_{j = n}^{\infty}\kappa_j\mu[x^j],
\end{equation}
but the choice of $\widecheck{\mu}[x^j]$ for $j = 0,\dots,M-1$ is arbitrary. 

For any choice of Geronimus transform $\widecheck{\mu}$, adding any complex measure supported on $w_1,\dots,w_M$ to $\widecheck{\mu}$ produces a new Geronimus transform satisfying~\eqref{eq:gt definition one functional}. If some zero $w_j$ has higher multiplicity we may also add derivatives of $\delta_{w_j}$. Note that for measures we had a canonical choice for the Geronimus transform with no additional point masses. 
For functionals there does not seem to be such a choice. In the one-step case one option would be 
\nm{eq:one-step geronimus functional}{\widecheck{\mu}[P(x)] = \mu\left[\frac{P(x)-P(z_1)}{x-z_1}\right],}
but we avoid any specific choice of Geronimus transform in this paper. 

\begin{defn}
    We refer to a rational perturbation of the functional $\mu$ as any functional $\widetilde{\mu}$ satisfying $\Psi\widetilde{\mu} = \Phi\mu$ for some non-zero polynomials $\Phi$ and $\Psi$.
\end{defn}
If $\Psi = 1$ we get Christoffel transforms and if $\Phi = 1$ we get Geronimus transforms. This setting extends the {general} rational perturbations of measures \eqref{eq:rp several measures}
to linear functionals. 
Uvarov's formula~\eqref{eq:rp thm one measure} can be extended to this general setting as follows. 
\begin{thm}\label{prop}
Let {$\mu$ be any quasi-definite functional, and } $\widetilde\mu$ be any functional satisfying $\Psi \widetilde{\mu} = \Phi \mu$ for polynomials $\Phi$ and $\Psi$ with simple zeros. Choose any Geronimus transform $\widecheck{\mu}$ of $\widetilde{\mu}$ corresponding to $\Phi$, that is, 
$\Phi\widecheck{\mu} = \widetilde{\mu}$. Finally, let 
\nm{eq:q rp one functional}{
Q_n(z) = \widecheck{\mu}\left[P_n(x)\frac{\Psi(x)}{x-z}\right], \qquad z = w_1,\dots,w_M.
}
Then $n \geq M$ is normal for $\widetilde{\mu}$ if and only if $D_n$ in \eqref{eq:dn rp one measure} is non-zero, and then~\eqref{eq:rp thm one measure} holds with this choice of $Q_n$.
\end{thm}
\begin{rem}
Note that $\frac{\Psi(x)}{x-w_k}$ is a polynomial, so that ~\eqref{eq:q rp one functional} is well-defined for any choice of functional $\widecheck{\mu}$. Also note that the functional $\widecheck{\mu}$ does not have to be quasi-definite. 
\end{rem}
\begin{rem}\label{rem:nonSimple}
    The case when $n < M$ can be handled by modifying the initial entries in the determinants~\eqref{eq:rp thm one measure} and~\eqref{eq:dn rp one measure}, similarly to Uvarov's \cite{Uvarov}. We work this out in the multiple orthogonality case in Sections \ref{type I formula}-\ref{type II formula}. The case when zeros of $\Phi$ and/or $\Psi$ are not simple can be handled in the standard way, just as in the usual Christoffel/Uvarov formulas. See Section~\ref{general case} for the details.
\end{rem}
\begin{proof}
    This is a special case $r=1$ of the more general Theorem~\ref{thm:type II thm} that we prove in Section~\ref{type II formula}.
\end{proof}

We stress that in the simplest case when:
\begin{equation}
    \begin{cases}
        (a) \mbox{  } \mu \mbox{ is a positive measure}; 
        
        \\
        (b) \mbox{  } \{w_k\}_{k=1}^M \cap \supp\mu =\varnothing;
        \\
        (c) \mbox{  } \widetilde{\mu}\{w_k\}=0 \mbox{ for all } k=1,\ldots, M,
    \end{cases}
\end{equation}
then $d\widecheck{\mu}(x) = \frac{1}{\Psi(x)} d\mu(x)$ and therefore
\begin{equation}
Q_n(w_k) = \int \frac{P_n(x)}{x-w_k} d\mu(x)
\end{equation}
is reduced to the function of the second kind evaluated at $w_k$. This is exactly the setting of the celebrated result of Uvarov~\cite{Uvarov}.


If any of the conditions (a)--(c) fail, the expression for $Q_n(w_k)$ is no longer this simple, and its explicit form depends on the particular setting under consideration. In this case, $\widecheck{\mu}$, and hence $Q_n$, can be determined via a straightforward calculation, as illustrated in Corollaries~\ref{cor:basic rp thm} and~\ref{cor:pureUvarov} (with further examples in Sections~\ref{examples} and ~\ref{ss:Laguerre}). For instance, if (a) and (b) hold but (c) fails, then $Q_n$ can take either the form~\eqref{eq:rp discrete parts one measure} in the case when $\Phi$ and $\Psi$ share no zeros, or \eqref{eq:pureUvarov} if $\Phi=\Psi$. Note that while both expressions involve functions of the second kind, the additional discrete terms vary greatly. In our opinion, it is remarkable that the expression~\eqref{eq:q rp one functional} is able to accommodate all the possible cases simultaneously, and as a consequence, Theorem~\ref{prop}  (combined with Remark~\ref{rem:nonSimple}) works for {\it any} rational perturbation of {\it any} quasi-definite linear functional.




\begin{cor}\label{cor:basic rp thm}
Assume $\Phi$ and $\Psi$ have non-overlapping simple zeros $\{z_j\}_{j=1}^N$ and $\{w_j\}_{j=1}^M$, respectively. Suppose $\Phi/\Psi$ is positive on $\supp(\mu)$. Let $\widetilde{\mu}$ be the measure satisfying
\nm{eq:rp discrete parts one measure}{\int f(x) d\widetilde{\mu}(x) = \int f(x) \frac{\Phi(x)}{\Psi(x)}d\mu(x) + \sum_{k = 1}^M c_k f({w_k})}
with $c_k>0$. Then~\eqref{eq:rp thm one measure} holds with
\nm{eq:q rp discrete parts one measure}{
Q_n(w_k) = \int \frac{P_n(x)}{x-w_k}d\mu(x) + \frac{c_k \Psi'(w_k)}{\Phi(w_k)} P_n(w_k), \qquad k = 1,\dots,M.}
\end{cor}
\begin{proof}
    Define
    \m{\int f(x)d\widecheck{\mu}(x) = \int f(x)\frac{1}{\Psi(x)}d\mu(x) + \sum_{k = 1}^{M}\frac{c_k}{\Phi(w_k)}f(w_k).}    
    

    Then \eqref{eq:q rp one functional} can be computed as
    $$
    Q_n(w_k) = \int f(x)\frac{P_n(x)}{x-w_k}d\mu(x) + \sum_{k = 1}^{M}\frac{c_k}{\Phi(w_k)} \left[P_n(x)\frac{\Psi(x)}{x-w_k}\right]_{x=w_k},
    $$
    which becomes \eqref{eq:q rp discrete parts one measure}.
\end{proof}

 

\begin{cor}\label{cor:pureUvarov}
Assume $\{w_k\}_{k=1}^M \subset\bbR\setminus\supp(\mu)$ are distinct and let 
\nm{eq:rp discrete parts one measureNEW}{\int f(x) d\widetilde{\mu}(x) = \int f(x) d\mu(x) + \sum_{k = 1}^M c_k f({w_k}),}
with $c_k>0$, $1 \leq k \leq N$. 
Then~\eqref{eq:rp thm one measure} holds with $w_j=z_j$, $1\le k \le M$, and 
\nm{eq:pureUvarov}{
Q_n(w_k) = \int \frac{P_n(x)}{x-w_k}d\mu(x) + c_k P_n'(w_k), \qquad k = 1,\dots,M.}
\end{cor}
\begin{rem}
    The result of Corollary \ref{cor:basic rp thm} appears in the work of Zhedanov \cite{Zhe97}. In~\cite[Sect 2]{Uvarov} Uvarov obtained another formula for $\widetilde{P}_n$ for the transformation~\eqref{eq:rp discrete parts one measureNEW} in terms of kernel polynomials. 
\end{rem}

\subsection{Multiple Orthogonal Polynomials of Type I and Type II}\label{moprl}\hfill\\

Let $r$ be a positive integer and consider a system of linear functionals $\bm{\mu} = (\mu_1,\dots,\mu_r)$. We assume that each of $\mu_1,\dots,\mu_r$ are quasi-definite throughout the rest of the paper. We remind the basics of multiple orthogonal polynomials. For more details, see \cite{Aptekarev,Ismail,Nikishin}. We write $\bm{n}$ for multi-indices $(n_1,\dots,n_r) \in \N^r$, and $\abs{\bm{n}} = n_1 + \cdots + n_r$. Here we use the notation $\N = \set{0,1,2,\dots}$.
\begin{defn}
    A type I multiple orthogonal polynomial for the multi-index $\bm{n}$ is a non-zero vector of polynomials $\bm{A}_{\bm{n}} = (A_{\bm{n}}^{(1)},\dots,A_{\bm{n}}^{(r)})$ satisfying the orthogonality condition
\nm{eq:orthogonality conditions type I}{\sum_{j = 1}^r \mu_j[A_{\bm{n}}^{(j)}(x) x^p] = 0,\qquad p = 0,1,\dots,\abs{\bm{n}}-2,}
and the degree conditions
\nm{eq:degree conditions type I}{& \deg{A_{\bm{n}}^{(1)}} \leq n_1 - 1, \\
& \qquad \qquad \vdots \\
& \deg{A_{\bm{n}}^{(r)}} \leq n_r - 1.} 
Note that $A_{\bm{n}}^{(j)} = 0$ is the only possible choice when $n_j = 0$. In the special case $\bm{n} = \bm{0}$ we allow $\bm A_{\bm{0}} = \bm{0}$ to be a type I polynomial. 
\end{defn}
\begin{defn}
    A type II multiple orthogonal polynomial for the multi-index $\bm{n}$ is a non-zero polynomial $P_{\bm{n}}$ satisfying the orthogonality conditions
\nm{eq:orthogonality conditions type II}{& \mu_1[P_{\bm{n}}(x)x^p] = 0, \qquad p = 0,\dots,n_1-1, \\
& \qquad \qquad \vdots \\
& \mu_r[P_{\bm{n}}(x)x^p] = 0, \qquad p = 0,\dots,n_r-1,}
and the degree condition
\nm{eq:degree conditions type II}{\deg{P_{\bm{n}}} \leq \abs{\bm{n}}.}
\end{defn}

In both definitions we have less equations than coefficients to solve for, so  there always exists non-zero type I and type II polynomials. The following result extends the notion of normality and quasi-definite functional to the multiple orthogonality setting (see \cite{Ismail} and our discussion in \cite[Sect 2.6]{ctpaper}). 

\begin{lem}\label{lem:normality lemma}
For any multi-index $\bm{n} \neq \bm{0}$, the following statements are equivalent.
    \begin{itemize}
        \item[\textnormal{(i)}] There is a unique type I polynomial $(A_{\bm{n}}^{(1)},\dots,A_{\bm{n}}^{(r)})$ such that $\sum_{j = 1}^r \mu_j[A_{\bm{n}}^{(j)}(x) x^{\abs{\bm{n}}-1}] = 1$.
        \item[\textnormal{(ii)}] There is a unique type II polynomial $P_{\bm{n}}$ with degree $\abs{\bm{n}}$ such that $P_{\bm{n}}$ is monic.
        \item[\textnormal{(iii)}] $\sum_{j = 1}^r \mu_j[A_{\bm{n},j}(x) x^{\abs{\bm{n}}-1}] \neq 0$ for every type I polynomial.
        \item [\textnormal{(iv)}] $\deg{P_{\bm{n}}} = \abs{\bm{n}}$ for every type II polynomial.
        \item[\textnormal{(v)}] $\deg{A_{\bm{n}+\bm{e}_j}^{(j)}} = n_j$ for every type I polynomial (for the index $\bm{n}+\bm{e}_j$).
        \item[\textnormal{(vi)}] If $n_j > 0$, then $\mu_j[P_{\bm{n}-\bm{e}_j}(x) x^{n_j-1}] \neq 0$ for every type II polynomial (for the index $\bm{n}-\bm{e}_j$).
        \item[\textnormal{(vii)}] $\det H_{\bm{n}} \ne 0$, where 
\begin{equation}
    \label{eq:moprl matrix}
H_{\bm{n}} = \begin{pmatrix}
\nu_0^{(1)} & \nu_1^{(1)} & \cdots & \nu_{\abs{\bm{n}}-1}^{(1)} \\
\nu_1^{(1)} & \nu_2^{(1)} & \cdots & \nu_{\abs{\bm{n}}}^{(1)} \\
\vdots & \vdots & \ddots & \vdots \\
\nu_{n_1-1}^{(1)} & \nu_{n_1}^{(1)} & \cdots & \nu_{\abs{\bm{n}}+n_1-2}^{(1)} \\
\hline
& & \vdots \\
\hline
\nu_0^{(r)} & \nu_1^{(r)} & \cdots & \nu_{\abs{\bm{n}}-1}^{(r)} \\
\nu_1^{(r)} & \nu_2^{(r)} & \cdots & \nu_{\abs{\bm{n}}}^{(r)} \\
\vdots & \vdots & \ddots & \vdots \\
\nu_{n_r-1}^{(r)} & \nu_{n_r}^{(r)} & \cdots & \nu_{\abs{\bm{n}}+n_r-2}^{(r)} \\
\end{pmatrix},
\end{equation}
and $\nu_n^{(j)}$ are the moments $\nu_n^{(j)} = \mu_j[x^n]$, $n \in \N$, $j = 1,\dots,r$.
    \end{itemize}
\end{lem}
\begin{defn}
    We say that an index $\bm{n} \neq \bm{0}$ is normal if any of the equivalent statements in Lemma \ref{lem:normality lemma} hold. For a normal index we refer to (i) and (ii) as the normalized polynomials. The index $\bm{n} = \bm{0}$ is always taken to be normal. If every index $\bm{n} \in \N^r$ is normal we say that the system $\bm{\mu}$ is perfect. Note that when $r = 1$ the perfect systems correspond to the quasi-definite moment functionals. 
\end{defn}

\subsection{Sequences of Linearly Independent Polynomials}\label{linear indepence of regular sequences}\hfill\\

We want to consider multiple orthogonal polynomials along various sequences of multi-indices in $\bbZ_+^r$. For our proofs in the next sections to work out we need linear independence of the polynomials along our chosen sequences. The following definition lists the most important sequences. 

\begin{defn}\label{def:forward sequences}
    Consider $\bm{n}, \bm{m} \in \N^{r}$ with $\bm{n} \leq \bm{m}$, and let $d = \abs{\bm{m}}-\abs{\bm{n}}$. 
    \begin{itemize}
        \item[\textnormal{(i)}] A sequence {of multi-indices} $\set{\bm{n}+\bm{s}_j}_{j = 0}^d$ is  {called} a path from $\bm{n}$ to $\bm{m}$ if for some $k_1,\dots,k_d \in \set{1,\dots,r}$,
        \nm{eq:path definition}{\bm{s}_0 = \bm{0}, \qquad \bm{s}_j = \bm{s}_{j-1} + \bm{e}_{k_j}, \qquad \bm{n}+\bm{s}_d = \bm{m}, \qquad j = 1,\dots,d.}
        \item[\textnormal{(ii)}] {We say that} $\set{\bm{n}+\bm{s}_j}_{j = 0}^d$ is an (increasing) frame from $\bm{n}$ towards $\bm{m}$ if it consists of {$\{\bm{n}+j\bm{e}_k: 0\le j\le m_k-n_k, 1\le k\le r\}$}
         ordered in such a way that $\abs{\bm{s}_i} \leq \abs{\bm{s}_j}$ if $i < j$, $j = 1,\dots,d$. 
        \item[\textnormal{(iii)}] {We say that} $\set{\bm{n}+\bm{s}_j}_{j = 0}^d$ is an admissible sequence from $\bm{n}$ towards $\bm{m}$ if for some $k_1,\dots,k_d \in \set{1,\dots,r}$,
        \nm{eq:regular sequence definition}{\bm{s}_0 = \bm{0}, \qquad (\bm{s}_i)_{k_j} < (\bm{s}_j)_{k_j}, \qquad i < j, \qquad \bm{n} \leq \bm{n} + \bm{s}_j \leq \bm{m}, \qquad j = 1,\dots,d.}
    \end{itemize}
    \end{defn}
\begin{rem}
    We stress that frames towards $\bm{m}$ do not pass through $\bm{m}$. Admissible sequences towards $\bm{m}$ may pass through $\bm{m}$ or not. For examples, see  Figure 1.
\end{rem}
\begin{rem}\label{rem:regular sequences}
    It is easy to verify that both paths from $\bm{n}$ to $\bm{m}$ and frames from $\bm{n}$ towards $\bm{m}$ are admissible sequences from $\bm{n}$ towards $\bm{m}$. 
    The main purpose of introducing \eqref{eq:regular sequence definition} is that it offers minimal restrictions on the sequences of multi-indices that suffice for our proofs. In particular, \eqref{eq:regular sequence definition} is the property that we use to prove linear independence for type I polynomials. 
    If a reader is not interested in generality, they can replace any admissible sequence with their favorite path, or a frame.
\end{rem}
    
    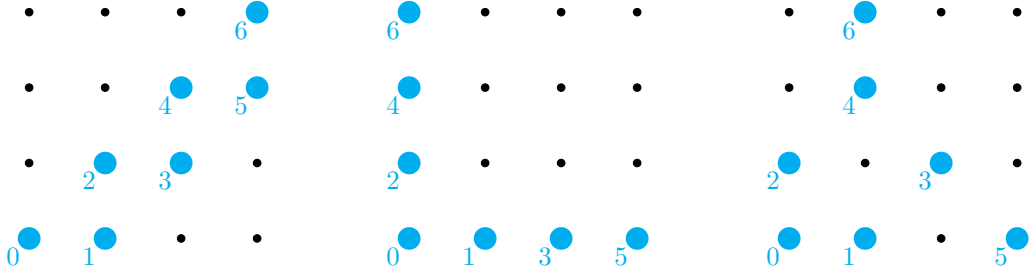
\begin{figure}[ht]
\begin{tikzpicture}\label{fig}

\foreach \x in {0,1,2,3,5,6,7,8,10,11,12,13}{
\foreach \y in {0,1,2,3}{
\node[draw,circle,inner sep=1pt,fill] at (\x,\y) {};}}
\fill[cyan] (0,0) circle[radius=1.5mm];
\fill[cyan] (1,0) circle[radius=1.5mm];
\fill[cyan] (1,1) circle[radius=1.5mm];
\fill[cyan] (2,1) circle[radius=1.5mm];
\fill[cyan] (2,2) circle[radius=1.5mm];
\fill[cyan] (3,2) circle[radius=1.5mm];
\fill[cyan] (3,3) circle[radius=1.5mm];

\draw[cyan] (0,0) node[anchor=north east] {$0$};
\draw[cyan] (1,0) node[anchor=north east] {$1$};
\draw[cyan] (1,1) node[anchor=north east] {$2$};
\draw[cyan] (2,1) node[anchor=north east] {$3$};
\draw[cyan] (2,2) node[anchor=north east] {$4$};
\draw[cyan] (3,2) node[anchor=north east] {$5$};
\draw[cyan] (3,3) node[anchor=north east] {$6$};

\fill[cyan] (5,0) circle[radius=1.5mm];
\fill[cyan] (6,0) circle[radius=1.5mm];
\fill[cyan] (7,0) circle[radius=1.5mm];
\fill[cyan] (8,0) circle[radius=1.5mm];
\fill[cyan] (5,1) circle[radius=1.5mm];
\fill[cyan] (5,2) circle[radius=1.5mm];
\fill[cyan] (5,3) circle[radius=1.5mm];

\draw[cyan] (5,0) node[anchor=north east] {$0$};
\draw[cyan] (6,0) node[anchor=north east] {$1$};
\draw[cyan] (5,1) node[anchor=north east] {$2$};
\draw[cyan] (7,0) node[anchor=north east] {$3$};
\draw[cyan] (5,2) node[anchor=north east] {$4$};
\draw[cyan] (8,0) node[anchor=north east] {$5$};
\draw[cyan] (5,3) node[anchor=north east] {$6$};

\fill[cyan] (10,0) circle[radius=1.5mm];
\fill[cyan] (11,0) circle[radius=1.5mm];
\fill[cyan] (10,1) circle[radius=1.5mm];
\fill[cyan] (12,1) circle[radius=1.5mm];
\fill[cyan] (11,2) circle[radius=1.5mm];
\fill[cyan] (13,0) circle[radius=1.5mm];
\fill[cyan] (11,3) circle[radius=1.5mm];

\draw[cyan] (10,0) node[anchor=north east] {$0$};
\draw[cyan] (11,0) node[anchor=north east] {$1$};
\draw[cyan] (10,1) node[anchor=north east] {$2$};
\draw[cyan] (12,1) node[anchor=north east] {$3$};
\draw[cyan] (11,2) node[anchor=north east] {$4$};
\draw[cyan] (13,0) node[anchor=north east] {$5$};
\draw[cyan] (11,3) node[anchor=north east] {$6$};

\end{tikzpicture} 
\caption{The following picture displays three different admissible sequences from $(0,0)$ towards $(3,3)$. To the left we have a path (the step-line). In the middle we have a frame from $(0,0)$ towards $(3,3)$. It is easy to verify from the picture that each of these sequences satisfy \eqref{eq:regular sequence definition}. For example, in the sequence to the right we have $k_1 = 1, k_2 = 2, k_3 = 1, k_4 = 2, k_5 = 1, k_6 = 2$.
}
\end{figure}

\begin{lem}\label{lem:linear independence modified paths type I}
    Assume $\bm{\mu}$ is a perfect system and $\bm{n},\bm{m}\in\N^r$. If $\bm{n} \leq \bm{m}$ and $\set{\bm{n}+\bm{s}_j}_{j=0}^{d}$ is an admissible sequence from $\bm{n}$ towards $\bm{m}$, then the set of type I vectors $\set{\bm{A}_{\bm{n}+\bm{s}_j}}_{j = 0}^d$ is linearly independent.
\end{lem}

\begin{proof}
    Consider the dependence equation
    \m{\sum_{j = 0}^{d}c_j\bm{A}_{\bm{n}+\bm{s}_j} = \bm{0},}
    and suppose $c_j = 0$ for $j > l$. By \eqref{eq:regular sequence definition} and Lemma \ref{lem:normality lemma}(v) we have
    \m{\deg{A_{\bm{n}+\bm{s}_j}^{(k_l)}} < \deg{A_{\bm{n}+\bm{s}_l}^{(k_l)}}, \qquad j < l,} 
    Hence $c_l = 0$, so we must have $c_0 = \dots = c_d = 0$, and linear independence follows.
\end{proof}

For linear independence of type II polynomials, we need admissible sequences that go backwards. 

\begin{defn}\label{def:backward sequences}
    Consider $\bm{n}, \bm{m} \in \N^{r}$ with $\bm{n} \geq \bm{m}$, and let $d =\abs{\bm{n}} -  \abs{\bm{m}}$. 
    \begin{itemize}
        \item[\textnormal{(i)}] {A} sequence {of multi-indices} $\set{\bm{n}-\bm{s}_j}_{j = 0}^d$ is {called} a path from $\bm{n}$ to $\bm{m}$ if for some $k_1,\dots,k_d \in \set{1,\dots,r}$,
        \nm{eq:path definition backwards}{\bm{s}_0 = \bm{0}, \qquad \bm{s}_j = \bm{s}_{j-1} + \bm{e}_{k_j}, \qquad \bm{n}-\bm{s}_d = \bm{m}, \qquad j = 1,\dots,d.}
        \item[\textnormal{(ii)}] {We say that } $\set{\bm{n}-\bm{s}_j}_{j = 0}^d$ is a (decreasing) frame from $\bm{n}$ towards $\bm{m}$ if it consists of {$\{\bm{n}-j\bm{e}_k: 0\le j\le n_k-m_k, 1\le k\le r\}$}
        ordered in such a way that $\abs{\bm{s}_i} \leq \abs{\bm{s}_j}$ if $i < j$, $j = 1,\dots,d$. 
        \item[\textnormal{(iii)}] {We say that} $\set{\bm{n}-\bm{s}_j}_{j = 0}^d$ is an admissible sequence from $\bm{n}$ towards $\bm{m}$ if for some $k_1,\dots,k_d \in \set{1,\dots,r}$,
        \nm{eq:regular sequence definition backwards}{\bm{s}_0 = \bm{0}, \qquad (\bm{s}_i)_{k_j} < (\bm{s}_j)_{k_j}, \qquad i < j, \qquad \bm{m} \leq \bm{n} - \bm{s}_j \leq \bm{n}, \qquad j = 1,\dots,d.}
    \end{itemize}
    \end{defn}

\begin{lem}\label{lem:linear independence modified paths type II}
    Assume $\bm{\mu}$ is a perfect system and $\bm{n},\bm{m}\in\N^r$. If $\bm{n} \geq \bm{m}$ and $\set{\bm{n}-\bm{s}_j}_{j=0}^{d}$ is an admissible sequence from $\bm{n}$ to $\bm{m}$, then the  set of type II polynomials $\set{P_{\bm{n}-\bm{s}_j}}_{j=0}^d$ is linearly independent.
\end{lem}

\begin{proof}
    Consider the dependence equation
    \m{\sum_{j = 0}^{d}c_jP_{\bm{n}-\bm{s}_j} = 0,}
    with $c_j = 0$ for $j > l$. By \eqref{eq:regular sequence definition backwards} and Lemma \ref{lem:normality lemma}(vi) we have
    \m{\mu_{k_l}[P_{\bm{n}-\bm{s}_l}(x)x^{(\bm{n}-\bm{s}_l)_{k_l}}] \neq 0, \qquad \mu_{k_l}[P_{\bm{n}-\bm{s}_j}(x)x^{(\bm{n}-\bm{s}_l)_{k_l}}] = 0, \qquad j < l.}
    Hence we get
    \m{0 = \mu_{k_l}\left[\sum_{j = 0}^{d}c_jP_{\bm{n}-\bm{s}_j}(x)x^{(\bm{n}-\bm{s}_l)_{k_l}}\right] = c_l\mu_{k_l}[P_{\bm{n}-\bm{s}_l}(x)x^{(\bm{n}-\bm{s}_l)_{k_l}}],}
    so we must have $c_l = 0$, and linear independence follows.
\end{proof}

\begin{rem}
    In \cite{ctpaper}, we used sequences $\set{\bm{n}+\bm{s}_j}_{j = 0}^{d}$ with $\abs{\bm{s}_j} = j$, $j = 0,\dots,d$. This definition is more straightforward than \eqref{eq:regular sequence definition} and \eqref{eq:regular sequence definition backwards}, but it is worse for us since it does not contain frames, which we view as the most natural sequences to use for the results of this paper (but it does still contain paths). For these sequences, the linear independence of type I and type II polynomials holds as well, but through (iii) and (iv) of Lemma \ref{lem:normality lemma} rather than (v) and (vi). These sequences are also not always admissible. 
\end{rem}

\section{Main Results}\label{the general formula}

\subsection{Rational Perturbations of Type I polynomials}\label{type I formula}\hfill\\

For a perfect system $\bm{\mu} = (\mu_1,\dots,\mu_r)$, we consider systems $\widetilde{\bm{\mu}} = (\widetilde{\mu}_1,\dots,\widetilde{\mu}_r)$ with $\Psi\widetilde{\mu}_j = \Phi_j\mu_j$ for polynomials $\Phi_j(x) = \prod_{k = 1}^{N_j}(x-z_{j,k})$ and $\Psi(x) = \prod_{k = 1}^{M}(x-w_k)$, $j = 1,\dots,r$. For now we assume that $\Phi_1,\dots,\Phi_r$ and $\Psi$ have only simple roots, but we allow the different polynomials to have common roots. 
We write $\bm{N}$ for the index $(N_1,\dots,N_r)$.

For the determinantal formula we need to consider Geronimus transforms $\widecheck{\bm{\mu}} = (\widecheck{\mu}_1,\dots,\widecheck{\mu}_r)$ of $\widetilde{\bm{\mu}}$, defined by $\Phi_j\widecheck{\mu}_j = \widetilde{\mu}_j$, $j = 1,\dots,r$. We define numbers $B_{\bm{n}}(w_1),\dots,B_{\bm{n}}(w_M)$, similarly to \eqref{eq:q rp one functional}, by 
\nm{eq:q type I}{B_{\bm{n}}(z) = \sum_{j = 1}^r\widecheck{\mu}_j\left[A_{\bm{n}}^{(j)}(x)\frac{\Psi(x)}{x-z}\right], \qquad z = w_1,\dots,w_M, \qquad \bm{n}\in\N^r.}
{We point out that in the simplest case when}
\begin{equation}
    \begin{cases}
        (a) \mbox{  } \mu_j \mbox{ is a positive measure};
        \\
        (b) \mbox{  } \{w_k\}_{k=1}^M \cap \supp\mu_j =\varnothing;
        \\
        (c) \mbox{  } \widetilde{\mu}_j\{w_k\}=0 \mbox{ for all } k=1,\ldots, M,
    \end{cases}
\end{equation}
hold true for every $j=1,\ldots, r$,
then $B_{\bm{n}}$ reduces to
\begin{equation}
    B_{\bm{n}}(z) = \sum_{j = 1}^r \int \frac{A_{\bm{n}}^{(j)}(x)}{x-z} d\mu_j(x).
\end{equation}
See also the discussion in the end of Section~\ref{moment functionals}, and Sections~\ref{examples},~\ref{ss:Laguerre} for other examples.

{To simplify the formulas below, we define the block matrices}
\nm{eq:type I blocks}{A_{\bm{n}}^{(j)}(\bm{z}_j) = \begin{pmatrix}
    A_{\bm{n}}^{(j)}(z_{j,1}) \\
    \vdots \\
    A_{\bm{n}}^{(j)}(z_{j,N_j})
\end{pmatrix}, \qquad B_{\bm{n}}(\bm{w}) = \begin{pmatrix}
    B_{\bm{n}}(w_1) \\
    \vdots \\
    B_{\bm{n}}(w_M)
\end{pmatrix}, \qquad j = 1,\dots,r, \qquad \bm{n} \in \N^r.}

\begin{thm}\label{thm:type I thm}
    Assume $\bm{\mu}$ is perfect. Let $\bm{n} \in \N^r$ with $\abs{\bm{n}} > M$. Consider an admissible sequence $\set{\bm{m}+\bm{s}_j}_{j = 0}^{\abs{\bm{N}}+M}$ from $\bm{m}$ towards $\bm{n}+\bm{N}$, for some $\bm{m} \in \N^r$ with $\abs{\bm{m}} = \abs{\bm{n}}-M$  and $\bm{m} \leq \bm{n}+\bm{N}$ (see Definition \ref{def:forward sequences} and Remark \ref{rem:regular sequences}). Then the determinantal expression
    \nm{eq:det formula rational type I}{\widetilde{A}_{\bm{n}}^{(j)}(x) = \Phi_j(x)^{-1}\det{\begin{pmatrix}
A_{\bm{m}}^{(j)}(x) & A_{\bm{m}+\bm{s}_{1}}^{(j)}(x) & \cdots & A_{\bm{m}+\bm{s}_{\abs{\bm{N}}+M}}^{(j)}(x) \\
A_{\bm{m}}^{(1)}(\bm{z}_1) & A_{\bm{m}+\bm{s}_{1}}^{(1)}(\bm{z}_1) & \cdots & A_{\bm{m}+\bm{s}_{\abs{\bm{N}}+M}}^{(1)}(\bm{z}_1) \\
\vdots & \vdots & \ddots & \vdots \\
A_{\bm{m}}^{(r)}(\bm{z}_{r}) & A_{\bm{m}+\bm{s}_{1}}^{(r)}(\bm{z}_r) & \cdots & A_{\bm{m}+\bm{s}_{\abs{\bm{N}}+M}}^{(r)}(\bm{z}_r) \\
B_{\bm{m}}(\bm{w}) & B_{\bm{m}+\bm{s}_{1}}(\bm{w}) & \cdots & B_{\bm{m}+\bm{s}_{\abs{\bm{N}}+M}}(\bm{w})
\end{pmatrix}}, \qquad j = 1\dots,r,}
defines a vector of polynomials $\bm{\widetilde{A}}_{\bm{n}}$ that satisfies the type I orthogonality condition \eqref{eq:orthogonality conditions type I} and the type I degree conditions \eqref{eq:degree conditions type I} for the index $\bm{n}$. The index $\bm{n}$ is normal for $\widetilde{\bm{\mu}}$ if and only if $D^{(I)}_{\bm{n}} \neq 0$, where
\nm{eq:dn rational type I}{D^{(I)}_{\bm{n}} = \det{\begin{pmatrix}
A_{\bm{m}+\bm{s}_{1}}^{(1)}(\bm{z}_1) & \cdots & A_{\bm{m}+\bm{s}_{\abs{\bm{N}}+M}}^{(1)}(\bm{z}_1) \\
\vdots & \ddots & \vdots \\
A_{\bm{m}+\bm{s}_{1}}^{(r)}(\bm{z}_r) & \cdots & A_{\bm{m}+\bm{s}_{\abs{\bm{N}}+M}}^{(r)}(\bm{z}_r) \\
B_{\bm{m}+\bm{s}_{1}}(\bm{w}) & \cdots & B_{\bm{m}+\bm{s}_{\abs{\bm{N}}+M}}(\bm{w})
\end{pmatrix}},}
and $\tfrac{1}{D^{(I)}_{\bm{n}}} \bm{\widetilde{A}}_{\bm{n}}$ is the normalized type I polynomial with respect to the system $\widetilde{\bm{\mu}}$. 

In the case $0 < \abs{\bm{n}} \leq M$ the same results hold if we put $\bm{A}_{\bm{m}+\bm{s}_j}(x) = \bm{0}$ and $B_{\bm{m}+\bm{s}_j}(z) = z^{M-\abs{\bm{n}}-j}$ for $j = 0,\dots,M-\abs{\bm{n}}$, and let $\set{\bm{m}+\bm{s}_j}_{M-\abs{\bm{n}}}^{\abs{\bm{N}}+M}$ be an admissible sequence from $\bm{0}$ towards $\bm{n}+\bm{N}$. If $\bm{n} = \bm{0}$ we have $\widetilde{\bm{A}}_{\bm{n}} = \bm{0}$.
\end{thm}
\begin{rem}
    This criterion, which states that $\bm{n}$ is normal if and only if $D^{(I)}_{\bm{n}} \neq 0$, improves our result~\cite[Thm 3.14]{ctpaper}, ~\cite[Thm 3.15]{ctpaper} for Christoffel transforms, where only one direction was established (normality of $\bm{n}$ implies $D^{(I)}_{\bm{n}} \neq 0$). 
\end{rem}
    \begin{figure}[ht]\label{type I fig}
\begin{tikzpicture}

\foreach \x in {0,1,2,3,5,6,7,8,10,11,12,13}{
\foreach \y in {0,1,2,3}{
\node[draw,circle,inner sep=1pt,fill] at (\x,\y) {};}}
\fill[cyan] (0,0) circle[radius=1.5mm];
\fill[cyan] (1,0) circle[radius=1.5mm];
\fill[cyan] (2,0) circle[radius=1.5mm];
\fill[cyan] (3,0) circle[radius=1.5mm];
\fill[cyan] (0,1) circle[radius=1.5mm];
\fill[cyan] (0,2) circle[radius=1.5mm];
\fill[cyan] (0,3) circle[radius=1.5mm];
\fill[blue] (1,1) circle[radius=0.9mm];
\fill[blue] (0,0) circle[radius=0.9mm];
\fill[blue] (3,3) circle[radius=0.9mm];

\draw[cyan] (0,0) node[anchor=north east] {$0$};
\draw[cyan] (1,0) node[anchor=north east] {$1$};
\draw[cyan] (0,1) node[anchor=north east] {$2$};
\draw[cyan] (2,0) node[anchor=north east] {$3$};
\draw[cyan] (0,2) node[anchor=north east] {$4$};
\draw[cyan] (3,0) node[anchor=north east] {$5$};
\draw[cyan] (0,3) node[anchor=north east] {$6$};
\draw[blue] (1,1) node[anchor=south west] {$\bm{n}$};
\draw[blue] (0,0) node[anchor=south west] {$\bm{m}$};
\draw[blue] (3,3) node[anchor=south west] {$\bm{n}+\bm{N}$};

\fill[cyan] (5,0) circle[radius=1.5mm];
\fill[cyan] (6,0) circle[radius=1.5mm];
\fill[cyan] (6,1) circle[radius=1.5mm];
\fill[cyan] (7,1) circle[radius=1.5mm];
\fill[cyan] (8,1) circle[radius=1.5mm];
\fill[cyan] (6,2) circle[radius=1.5mm];
\fill[cyan] (6,3) circle[radius=1.5mm];
\fill[blue] (6,1) circle[radius=0.9mm];
\fill[blue] (5,0) circle[radius=0.9mm];
\fill[blue] (8,3) circle[radius=0.9mm];

\draw[cyan] (5,0) node[anchor=north east] {$0$};
\draw[cyan] (6,0) node[anchor=north east] {$1$};
\draw[cyan] (6,1) node[anchor=north east] {$2$};
\draw[cyan] (7,1) node[anchor=north east] {$3$};
\draw[cyan] (6,2) node[anchor=north east] {$4$};
\draw[cyan] (8,1) node[anchor=north east] {$5$};
\draw[cyan] (6,3) node[anchor=north east] {$6$};
\draw[blue] (6,1) node[anchor=south west] {$\bm{n}$};
\draw[blue] (5,0) node[anchor=south west] {$\bm{m}$};
\draw[blue] (8,3) node[anchor=south west] {$\bm{n}+\bm{N}$};

\fill[cyan] (10,0) circle[radius=1.5mm];
\fill[cyan] (11,0) circle[radius=1.5mm];
\fill[cyan] (11,1) circle[radius=1.5mm];
\fill[cyan] (12,1) circle[radius=1.5mm];
\fill[cyan] (13,1) circle[radius=1.5mm];
\fill[cyan] (13,2) circle[radius=1.5mm];
\fill[cyan] (13,3) circle[radius=1.5mm];
\fill[blue] (11,1) circle[radius=0.9mm];
\draw[blue] (11,1) node[anchor=south west] {$\bm{n}$};
\fill[blue] (10,0) circle[radius=0.9mm];
\fill[blue] (13,3) circle[radius=0.9mm];

\draw[cyan] (10,0) node[anchor=north east] {$0$};
\draw[cyan] (11,0) node[anchor=north east] {$1$};
\draw[cyan] (11,1) node[anchor=north east] {$2$};
\draw[cyan] (12,1) node[anchor=north east] {$3$};
\draw[cyan] (13,1) node[anchor=north east] {$4$};
\draw[cyan] (13,2) node[anchor=north east] {$5$};
\draw[cyan] (13,3) node[anchor=north east] {$6$};
\draw[blue] (10,0) node[anchor=south west] {$\bm{m}$};
\draw[blue] (13,3) node[anchor=south west] {$\bm{n}+\bm{N}$};

\end{tikzpicture} 
\caption{The following picture displays three different sequences that can be considered in Theorem \ref{thm:type I thm} in the case $r = 2$ with $\deg{\Phi_1} = \deg{\Phi_2} = \deg{\Psi} = 2$, and the choice $\bm{m} = (n_1-1,n_2-1)$. To the left we have 
a frame from $\bm{m}$ towards $\bm{n}+\bm{N}$. We consider this to be a natural choice, even though the sequence does not include $\bm{n}$ (although in the most important case $\Psi = 1$ the sequence would start at $\bm{n}$). Each of the three sequences here have natural generalizations to any choice of $\Phi_1,\dots,\Phi_r$ and $\Psi$. 
}
\end{figure}
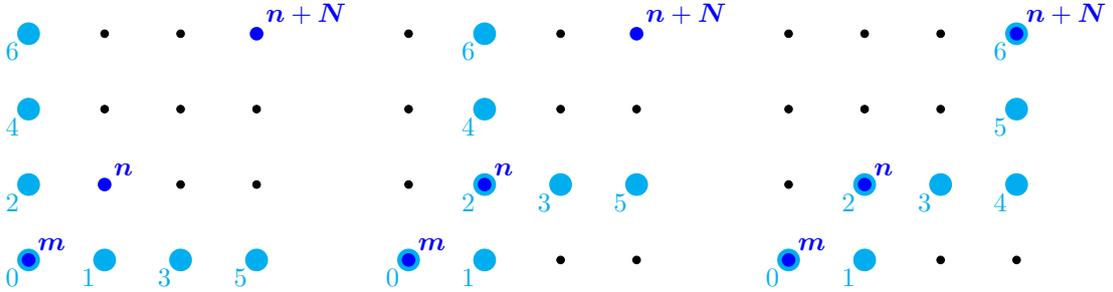

\begin{proof}
    We first prove the result in the case $\abs{\bm{n}} > M$. 
    {Note that the right-hand side of~\eqref{eq:det formula rational type I} is a polynomial. Denote it by $\widetilde{A}_{\bm{n}}^{(j)}$, and let us show that these polynomials are indeed the type I multiple orthogonal polynomials for $\widetilde{\bm{\mu}}$.} Since each element of the sequence $\set{\bm{m}+\bm{s}_j}_{j = 0}^{\abs{\bm{N}}+M}$ is below $\bm{n}+\bm{N}$, we have $\deg{\widetilde{A}_{\bm{n}}^{(j)}} \leq (n_j+N_j-1)-N_j = n_j - 1$ for each $j = 1,\dots,r$. Since the sequence is above $\bm{m}$, we have, by $|\bm{m}| = |\bm{n}|-M$, $\Psi\widetilde{\mu}_j = \Phi_j{\mu}_j$ and \eqref{eq:det formula rational type I},  
    \m{\sum_{j = 1}^{r}\widetilde{\mu}_j[\widetilde{A}_{\bm{n}}^{(j)}(x)\Psi(x)x^p] = \sum_{j = 1}^{r}\mu_j[\Phi_j(x)\widetilde{A}_{\bm{n}}^{(j)}(x)x^p] = 0, \qquad p = 0,\dots,\abs{\bm{n}}-M-2.}
    If $p < M$, we apply the partial fraction decomposition to $x^p/\Psi(x)$ to write
    \nm{eq:type I partial fraction}{x^p = \sum_{k = 1}^M \alpha_{p,k} \frac{\Psi(x)}{x-w_k}.}
    From this we get
    \m{\sum_{j = 1}^{r}\widetilde{\mu}_j[\widetilde{A}_{\bm{n}}^{(j)}(x)x^p] = \sum_{k = 1}^M\alpha_{p,k}\sum_{j = 1}^{r}  \widetilde{\mu}_j\left[\widetilde{A}_{\bm{n}}^{(j)}(x)\frac{\Psi(x)}{x-w_k}\right] = \sum_{k = 1}^M \alpha_{p,k} \sum_{j = 1}^{r}\widecheck{\mu}_j\left[\Phi_j(x)\widetilde{A}_{\bm{n}}^{(j)}(x)\frac{\Psi(x)}{x-w_k}\right] = 0,}
    since we get the determinant in \eqref{eq:det formula rational type I} but with the first row replaced by a linear combination of the last $M$ rows. Since every polynomial of degree $\leq \abs{\bm{n}} - 2$ is in the span of $\set{x^p}_{p=0}^{M-1}\cup\set{\Psi(x)x^p}_{p=0}^{\abs{\bm{n}}-M-2}$, we end up with
    \m{\sum_{j = 1}^{r}\widetilde{\mu}_j[\widetilde{A}_{\bm{n}}^{(j)}(x)x^p] = 0, \qquad p = 0,\dots,\abs{\bm{n}}-2.}
    The sequence starts at $\bm{m} \neq \bm{0}$ and increases, so we also have
    \m{\sum_{j = 1}^{r}\widetilde{\mu}_j[\widetilde{A}_{\bm{n}}^{(j)}(x)x^{\abs{\bm{n}}-1}] = \sum_{j = 1}^{r}\widetilde{\mu}_j[\widetilde{A}_{\bm{n}}^{(j)}(x)\Psi(x)x^{\abs{\bm{n}}-M-1}] = \sum_{j = 1}^{r}\mu_j[\Phi_j(x)\widetilde{A}_{\bm{n}}^{(j)}(x)x^{\abs{\bm{n}}-M-1}] = D^{(I)}_{\bm{n}}.}
    
    It remains to prove that $D^{(I)}_{\bm{n}} \neq 0$ if and only if $\bm{n}$ is normal for $\widetilde{\bm{\mu}}$. We have $D^{(I)}_{\bm{n}} = 0$ if and only if there are constants $c_1,\dots,c_{\abs{\bm{N}}+M}$, not all zero, such that
\begin{align}
\label{eq:zerosDivisibility1}
& \sum_{i = 1}^{\abs{\bm{N}}+M} c_iA_{\bm{m}+\bm{s}_{i}}^{(j)}(x) = 0, \qquad x = z_{j,1},\dots,z_{j,N_j}, \qquad j = 1,\dots,r, \\\label{eq:QzerosDivisibility1}
& \sum_{i = 1}^{\abs{\bm{N}}+M} c_iB_{\bm{m}+\bm{s}_{i}}(x) = 0, \qquad x = w_1,\dots,w_M.
\end{align}
Now suppose $D^{(I)}_{\bm{n}} = 0$ and choose $c_1,\dots,c_{N+\abs{\bm{M}}}$ as above. Consider the vector $\bm{A} = (A^{(1)},\dots,A^{(r)})$ given by
\nm{eq:general solution vector type I}{A^{(j)}(x) = \Phi_j(x)^{-1}\sum_{i = 1}^{\abs{\bm{N}}+M} c_i A_{\bm{m}+\bm{s}_{i}}^{(j)}(x), \qquad j = 1,\dots,r.}
{By~\eqref{eq:zerosDivisibility1}, each $A^{(j)}$ is a polynomial.} Since our sequence of indices is admissible, Lemma \ref{lem:linear independence modified paths type I} implies that $\bm{A} \not\equiv 0$. Since every index in the sum is below $\bm{n}+\bm{N}$ we have $\deg{A}^{(j)} \leq  n_j - 1$ for each $j = 1,\dots,r$. Since every index is also strictly above $\bm{m}$ we have, by similar arguments as before,
\nm{eq:Dn zero type I}{\sum_{j = 1}^r \widetilde{\mu}_j[A^{(j)}(x)x^p] = 0, \qquad p = 0,\dots,\abs{\bm{n}}-1.}
Hence $\bm{A}$ is a type I polynomial for the index $\bm{n}$ with respect to $\widetilde{\bm{\mu}}$, so $\bm{n}$ cannot be normal for $\widetilde{\bm{\mu}}$.

Conversely, suppose $\bm{n}$ is not normal for $\widetilde{\bm{\mu}}$. Then there is some vector $\bm{A} = (A^{(1)},\dots,A^{(r)}) \not\equiv \bm{0}$ {with $\deg A^{(j)}\le n_j-1$} such that \eqref{eq:Dn zero type I} holds. As we have seen, this implies
\begin{equation*}
    \sum_{j = 1}^r\mu_j[\Phi_j(x)A^{(j)}(x)x^p] = 0, \qquad p = 0,\dots,\abs{\bm{n}}-M-1.
\end{equation*}

{Note that $\deg \Phi_j(x)A^{(j)}(x) \le n_j+N_j-1$.}
For vectors $\bm{C} = (C^{(1)},\dots,C^{(r)})$ of polynomials $C^{(j)}(x) = \sum_{k = 0}^{n_j+N_j-1}\kappa_{j,k}x^{k}$, $j = 1,\dots,r$, consider the orthogonality relations
\begin{equation}\label{eq:anotherOne}
\sum_{j = 1}^r\mu_j[C^{(j)}(x)x^p] = 0, \qquad p = 0,\dots,\abs{\bm{n}}-M-1.
\end{equation}
Since $\abs{\bm{n}}-M=\abs{\bm{m}}$, this is an $\abs{\bm{m}}\times(\abs{\bm{n}}+\abs{\bm{N}})$ system of equations. The coefficient matrix of this system is given by $H_{\bm{m}}$ with $\abs{\bm{n}}+\abs{\bm{N}}-\abs{\bm{m}}=\abs{\bm{N}}+M$ columns added. Since $\bm{m}$ is normal for $\bm{\mu}$ the rank of this system is equal to $\abs{\bm{m}}$, so the solution space has dimension $\abs{\bm{N}}+M$. Note that each $\bm{A}_{\bm{m}+\bm{s}_i}$, $i = 1,\dots,\abs{\bm{N}}+M$, has degree $\le n_j+N_j-1$ and solves~\eqref{eq:anotherOne}. These vectors are linearly independent, by Lemma \ref{lem:linear independence modified paths type I}, so every solution is a linear combination of these vectors. The non-zero vector $(\Phi_1A^{(1)},\dots,\Phi_rA^{(r)})$ is also a solution, so we get \eqref{eq:general solution vector type I} for some $(c_1,\dots,c_{\abs{\bm{N}}+M}) \neq (0,\dots,0)$. This implies \eqref{eq:zerosDivisibility1} since $\Phi_j(z_{j,k}) = 0$. We also get \eqref{eq:QzerosDivisibility1}, since 
\begin{equation*}
    \sum_{i = 1}^{\abs{\bm{N}}+M} c_iB_{\bm{m}+\bm{s}_{i}}(w_k) = \sum_{j = 1}^r \widecheck{\mu}_j\left[\Phi_j(x)A^{(j)}(x)\frac{\Psi(x)}{x-w_k} \right] = \sum_{j = 1}^r \widetilde{\mu}_j\left[A^{(j)}(x)\frac{\Psi(x)}{x-w_k} \right] = 0,
\end{equation*}
by \eqref{eq:type I blocks} and $M-1 \leq \abs{\bm{n}}-2$. Since \eqref{eq:zerosDivisibility1}-\eqref{eq:QzerosDivisibility1} holds, we must have $D^{(I)}_{\bm{n}} = 0$. 

Now consider the case $\abs{\bm{n}} \leq M$. We want to check the orthogonality conditions with respect to $\widetilde{\bm{\mu}}$ against $x^p$ with $p \leq \abs{\bm{n}}-1$. Hence $p < M$, so we again consider the decomposition \eqref{eq:type I partial fraction}. We use the identity
\nm{eq:identity}{\frac{x^q}{x-z} = x^{q-1} + zx^{q-2} + \cdots + z^{q-1} + \frac{z^q}{x-z},}
substituted into \eqref{eq:type I partial fraction} multiplied with $x^q$ to get
\m{\frac{x^{p+q}}{\Psi(x)} = x^{q-1}\sum_{k = 1}^{M}\alpha_{p,k} + x^{q-2}\sum_{k = 1}^{M}\alpha_{p,k} w_k + \cdots + \sum_{k = 1}^{M}\alpha_{p,k} w_k^{q-1} + \sum_{k = 1}^{M}\alpha_{p,k}\frac{w_k^q}{x-w_k}.}
For $q = M - p$ the polynomial part of the left hand side is equal to one. Comparing with the right hand side we obtain
\nm{eq:magic type I}{\sum_{k = 1}^{M}\alpha_{p,k} w_k^q = \begin{cases}
    0, \qquad q = 0,\dots,M-p-2, \\
    1, \qquad q = M-p-1.
\end{cases}.}
Again, since $p < M$, we can write
\m{\sum_{j = 1}^{r}\widetilde{\mu}_j[\widetilde{A}_{\bm{n}}^{(j)}(x)x^p] = \sum_{k = 1}^M\alpha_{p,k} \sum_{j = 1}^{r}\widecheck{\mu}_j\left[\Phi_j(x)\widetilde{A}_{\bm{n}}^{(j)}(x)\frac{\Psi(x)}{x-w_k}\right]}
to get the determinant in \eqref{eq:det formula rational type I} but with each of the last $\abs{\bm{N}} + \abs{\bm{n}}$ elements in the first row replaced with a linear combination of each of the last $\abs{\bm{N}} + \abs{\bm{n}}$ elements of the last $M$ rows. Use this to get the last $\abs{\bm{N}} + \abs{\bm{n}}$ elements in the first row of the determinant equal to $0$ through elementary row operations. Then we get $-\sum_{k = 1}^{M}\alpha_{p,k} w_k^{M-\abs{n}-j}$ in entry $j+1$ of the first row for $j = 0,\dots,M-\abs{\bm{n}}$. For $p \leq \abs{\bm{n}}-2$ this is $0$ by \eqref{eq:magic type I}, since here $M-p-2 \geq M-\abs{\bm{n}}$, which means that all the entries in the first row of the determinant are now $0$. Hence we get
\m{\sum_{j = 1}^{r}\widetilde{\mu}_j[\widetilde{A}_{\bm{n}}^{(j)}(x)x^p] = 0, \qquad p = 0,\dots,\abs{\bm{n}}-2.}
For $p = \abs{\bm{n}}-1$ the same argument shows that the first row is now equal to $\begin{pmatrix}
    1 & 0 & \cdots & 0
\end{pmatrix}$. Now expand along the first row, and since the other rows remain unchanged from \eqref{eq:det formula rational type I} we get
\m{\sum_{j = 1}^{r}\widetilde{\mu}_j[\widetilde{A}_{\bm{n}}^{(j)}(x)x^p] = D^{(I)}_{\bm{n}}.}

{Similar arguments then show that $D^{(I)}_{\bm{n}} \neq 0$ when $\bm{n}$ is normal for $\widetilde{\bm{\mu}}$.}
The main difference is in showing that $\bm{A}$ in~\eqref{eq:general solution vector type I} is non-zero. If $\bm{A} = \bm{0}$, then by Lemma~\ref{lem:linear independence modified paths type I} we get $c_i=0$ for $i> M-\abs{\bm{n}}$.
Since the coefficient matrix of the system of equations \eqref{eq:QzerosDivisibility1} with respect to the unknowns $c_1,\ldots,c_{M-\abs{\bm{n}}}$ is a Vandermonde-type $M\times (M-|\bm{n}|)$ matrix, with at least as many rows as columns, we obtain $c_1 = \cdots = c_{M-\abs{\bm{n}}} = 0$. Hence $\bm{A}\ne \bm{0}$ in~\eqref{eq:general solution vector type I}. The rest of the proof works out similarly. 
\end{proof}

\subsection{Rational Perturbations of Type II polynomials}\label{type II formula}\hfill\\

For the type II case, we consider systems $\widetilde{\bm{\mu}} = (\widetilde{\mu}_1,\dots,\widetilde{\mu}_r)$ where $\Psi_j\widetilde{\mu}_j = \Phi\mu_j$ for polynomials $\Phi(x) = \prod_{k = 1}^{N}(x-z_{k})$ and $\Psi_j(x) = \prod_{k = 1}^{M_j}(x-w_{j,k})$, $j = 1,\dots,r$. Again we assume these polynomials only have simple zeros, and let $\bm{M} = (M_1,\dots,M_r)$.

The Geronimus transforms $\widecheck{\bm{\mu}} = (\widecheck{\mu}_1,\dots,\widecheck{\mu}_r)$ are now defined by $\Phi\widecheck{\bm{\mu}} = \widetilde{\bm{\mu}}$, and we define vectors of numbers $\bm{Q}_{\bm{n}}(z) = (Q_{\bm{n}}^{(1)}(z),\dots,Q_{\bm{n}}^{(r)}(z))$ by 
\nm{eq:q type II}{Q_{\bm{n}}^{(j)}(z) = \widecheck{\mu}_j\left[P_{\bm{n}}(x)\frac{\Psi_j(x)}{x-z}\right], \qquad z = w_{j,1},\dots,w_{j,M_j}, \qquad j = 1,\dots,r, \qquad \bm{n} \in \N^r.}

{We point out that in the simplest case when, for some $j=1,\ldots,r$,}
\begin{equation}
    \begin{cases}
        (a) \mbox{  } \mu_j \mbox{ is a positive measure};
        \\
        (b) \mbox{  } \{w_{j,k}\}_{k=1}^{M_j} \cap \supp\mu_j =\varnothing;
        \\
        (c) \mbox{  } \widetilde{\mu}_j\{w_{j,k}\}=0 \mbox{ for all } k=1,\ldots, M_j,
    \end{cases}
\end{equation}
then $Q^{(j)}_{\bm{n}}$ reduces to
\begin{equation}
    Q^{(j)}_{\bm{n}}(z) = \int \frac{P_{\bm{n}}(x)}{x-z} d\mu_j(x).
\end{equation}
See also the discussion in the end of Section~\ref{moment functionals} and Sections~\ref{examples},~\ref{ss:Laguerre} for other examples.

We will make use of the block matrices
\nm{eq:type II blocks}{P_{\bm{n}}(\bm{z}) = \begin{pmatrix}
    P_{\bm{n}}(z_1) \\
    \vdots \\
    P_{\bm{n}}(z_N)
\end{pmatrix}, \qquad Q_{\bm{n}}^{(j)}(\bm{w}_j) = \begin{pmatrix}
    Q_{\bm{n}}^{(j)}(w_{j,1}) \\
    \vdots \\
    Q_{\bm{n}}^{(j)}(w_{j,M_j})
\end{pmatrix}, \qquad j = 1,\dots,r, \qquad \bm{n} \in \N^r.}
\begin{thm}\label{thm:type II thm}
    Assume $\bm{\mu}$ is perfect. Let $\bm{n} \in \N^r$ with $\bm{n} \geq \bm{M}$. Consider an admissible  sequence $\set{\bm{m}-\bm{s}_j}_{j = 0}^{N+\abs{\bm{M}}}$ from $\bm{m}$ towards $\bm{n}-\bm{M}$, for some $\bm{m} \in \N^r$ with  $\abs{\bm{m}} = \abs{\bm{n}}+N$  and  $\bm{m} \geq \bm{n}-\bm{M}$ (see Section \ref{linear indepence of regular sequences}). Then the polynomial
    \nm{eq:det formula rational type II}{\widetilde{P}_{\bm{n}}(x) = \Phi(x)^{-1}\det{\begin{pmatrix}
P_{\bm{m}}(x) & P_{\bm{m}-\bm{s}_{1}}(x) & \cdots & P_{\bm{m}-\bm{s}_{N+\abs{\bm{M}}}}(x) \\
P_{\bm{m}}(\bm{z}) & P_{\bm{m}-\bm{s}_{1}}(\bm{z}) & \cdots & P_{\bm{m}-\bm{s}_{N+\abs{\bm{M}}}}(\bm{z}) \\
Q_{\bm{m}}^{(1)}(\bm{w}_1) & Q_{\bm{m}-\bm{s}_{1}}^{(1)}(\bm{w}_1) & \cdots & Q_{\bm{m}-\bm{s}_{N+\abs{\bm{M}}}}^{(1)}(\bm{w}_1) \\
\vdots & \vdots & \ddots & \vdots \\
Q_{\bm{m}}^{(r)}(\bm{w}_r) & Q_{\bm{m}-\bm{s}_{1}}^{(r)}(\bm{w}_r) & \cdots & Q_{\bm{m}-\bm{s}_{N+\abs{\bm{M}}}}^{(r)}(\bm{w}_r) \\
\end{pmatrix}}}
     satisfies the type II orthogonality conditions \eqref{eq:orthogonality conditions type II} and the type II degree condition \eqref{eq:degree conditions type II} for the index $\bm{n}$. The index $\bm{n}$ is normal for $\widetilde{\bm{\mu}}$ if and only if $D^{(II)}_{\bm{n}} \neq 0$, where
     \nm{eq:dn rational type II}{D^{(II)}_{\bm{n}} = \det{\begin{pmatrix}
P_{\bm{m}-\bm{s}_{1}}(\bm{z}) & \cdots & P_{\bm{m}-\bm{s}_{N+\abs{\bm{M}}}}(\bm{z}) \\
Q_{\bm{m}-\bm{s}_{1}}^{(1)}(\bm{w}_1) & \cdots & Q_{\bm{m}-\bm{s}_{N+\abs{\bm{M}}}}^{(1)}(\bm{w}_1) \\
\vdots & \ddots & \vdots \\
Q_{\bm{m}-\bm{s}_{1}}^{(r)}(\bm{w}_r) & \cdots & Q_{\bm{m}-\bm{s}_{N+\abs{\bm{M}}}}^{(r)}(\bm{w}_r) \\
\end{pmatrix}},}
and then $\tfrac{1}{D^{(II)}_{\bm{n}}} P_{\bm{n}}$ is the (unique) normalized type II polynomial with respect to the system $\widetilde{\bm{\mu}}$.

In the case where $n_j < M_j$ for some $j$, let $\set{\bm{m}-\bm{s}_j}_{j = 0}^{N + \abs{\bm{M}^*}}$ be an admissible sequence from $\bm{m}$ towards $\bm{n}-\bm{M}^*$, where $M^* = (M_1^*,\dots,M_r^*)$ is defined by $M_j^* = n_j$ if $n_j < M_j$, and $M_j^* = M_j$ if $n_j \geq M_j$, for $j = 1,\dots,r$. The same results hold if we put $P_{\bm{m}-\bm{s}_j} = 0$ for $j = N + \abs{\bm{M}^*}+1,\dots,N+\abs{\bm{M}}$, and
\nm{eq:q type II small indices}{&\begin{pmatrix}
    Q_{\bm{m}-\bm{s}_{N + \abs{\bm{M}^*}+1}}^{(1)}(\bm{w}_1) & \cdots & Q_{\bm{m}-\bm{s}_{N+\abs{\bm{M}}}}^{(1)}(\bm{w}_1) \\
    \vdots & \ddots & \vdots \\
    Q_{\bm{m}-\bm{s}_{N + \abs{\bm{M}^*}+1}}^{(r)}(\bm{w}_r) & \cdots & Q_{\bm{m}-\bm{s}_{N+\abs{\bm{M}}}}^{(r)}(\bm{w}_r)
\end{pmatrix} 
= \operatorname{diag}(W_1,\ldots,W_r),
\\ \\ & \qquad \textnormal{where} \qquad W_j = \begin{pmatrix}
    1 & w_{j,1} & \cdots & w_{j,1}^{M_j-n_j-1} \\
    \vdots & \vdots & \ddots & \vdots \\
    1 & w_{j,M_j} & \cdots & w_{j,M_j}^{M_j-n_j-1}
\end{pmatrix}, \qquad j = 1,\dots,r.}
{Here} $\operatorname{diag}(W_1,\dots,W_r)$ stands for the matrix with blocks $W_1,\dots,W_r$, with all other entries equal to zero. If $n_j \geq M_j$ then $W_j$ is empty and no block corresponding to $j$ is added to $\operatorname{diag}(W_1,\dots,W_r)$.
\end{thm}
\begin{rem}
    The criterion stating that $\bm{n}$ is normal if and only if $D^{(II)}_{\bm{n}} \neq 0$ is a minor improvement over our result~\cite[Thm 3.2(ii)]{ctpaper} for Christoffel transforms. In that work we required $D^{(II)}_{\bm{n}} \neq 0$ along a path of multi-indices leading to $\bm{n}$, which is more restrictive than necessary.
\end{rem}
    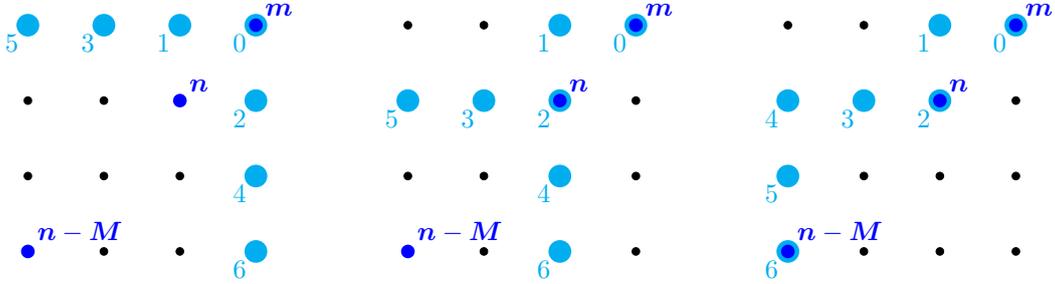
\begin{figure}[ht]
\begin{tikzpicture}

\foreach \x in {0,1,2,3,5,6,7,8,10,11,12,13}{
\foreach \y in {0,1,2,3}{
\node[draw,circle,inner sep=1pt,fill] at (\x,\y) {};}}
\fill[cyan] (3,3) circle[radius=1.5mm];
\fill[cyan] (2,3) circle[radius=1.5mm];
\fill[cyan] (3,2) circle[radius=1.5mm];
\fill[cyan] (1,3) circle[radius=1.5mm];
\fill[cyan] (3,1) circle[radius=1.5mm];
\fill[cyan] (0,3) circle[radius=1.5mm];
\fill[cyan] (3,0) circle[radius=1.5mm];
\fill[blue] (2,2) circle[radius=0.9mm];
\fill[blue] (0,0) circle[radius=0.9mm];
\fill[blue] (3,3) circle[radius=0.9mm];

\draw[cyan] (3,3) node[anchor=north east] {$0$};
\draw[cyan] (2,3) node[anchor=north east] {$1$};
\draw[cyan] (3,2) node[anchor=north east] {$2$};
\draw[cyan] (1,3) node[anchor=north east] {$3$};
\draw[cyan] (3,1) node[anchor=north east] {$4$};
\draw[cyan] (0,3) node[anchor=north east] {$5$};
\draw[cyan] (3,0) node[anchor=north east] {$6$};
\draw[blue] (2,2) node[anchor=south west] {$\bm{n}$};
\draw[blue] (0,0) node[anchor=south west] {$\bm{n}-\bm{M}$};
\draw[blue] (3,3) node[anchor=south west] {$\bm{m}$};

\fill[cyan] (8,3) circle[radius=1.5mm];
\fill[cyan] (7,3) circle[radius=1.5mm];
\fill[cyan] (7,2) circle[radius=1.5mm];
\fill[cyan] (6,2) circle[radius=1.5mm];
\fill[cyan] (7,1) circle[radius=1.5mm];
\fill[cyan] (5,2) circle[radius=1.5mm];
\fill[cyan] (7,0) circle[radius=1.5mm];
\fill[blue] (7,2) circle[radius=0.9mm];
\fill[blue] (5,0) circle[radius=0.9mm];
\fill[blue] (8,3) circle[radius=0.9mm];

\draw[cyan] (8,3) node[anchor=north east] {$0$};
\draw[cyan] (7,3) node[anchor=north east] {$1$};
\draw[cyan] (7,2) node[anchor=north east] {$2$};
\draw[cyan] (6,2) node[anchor=north east] {$3$};
\draw[cyan] (7,1) node[anchor=north east] {$4$};
\draw[cyan] (5,2) node[anchor=north east] {$5$};
\draw[cyan] (7,0) node[anchor=north east] {$6$};
\draw[blue] (7,2) node[anchor=south west] {$\bm{n}$};
\draw[blue] (5,0) node[anchor=south west] {$\bm{n}-\bm{M}$};
\draw[blue] (8,3) node[anchor=south west] {$\bm{m}$};

\fill[cyan] (13,3) circle[radius=1.5mm];
\fill[cyan] (12,3) circle[radius=1.5mm];
\fill[cyan] (12,2) circle[radius=1.5mm];
\fill[cyan] (11,2) circle[radius=1.5mm];
\fill[cyan] (10,2) circle[radius=1.5mm];
\fill[cyan] (10,1) circle[radius=1.5mm];
\fill[cyan] (10,0) circle[radius=1.5mm];
\fill[blue] (12,2) circle[radius=0.9mm];
\fill[blue] (10,0) circle[radius=0.9mm];
\fill[blue] (13,3) circle[radius=0.9mm];

\draw[cyan] (13,3) node[anchor=north east] {$0$};
\draw[cyan] (12,3) node[anchor=north east] {$1$};
\draw[cyan] (12,2) node[anchor=north east] {$2$};
\draw[cyan] (11,2) node[anchor=north east] {$3$};
\draw[cyan] (10,2) node[anchor=north east] {$4$};
\draw[cyan] (10,1) node[anchor=north east] {$5$};
\draw[cyan] (10,0) node[anchor=north east] {$6$};
\draw[blue] (12,2) node[anchor=south west] {$\bm{n}$};
\draw[blue] (10,0) node[anchor=south west] {$\bm{n}-\bm{M}$};
\draw[blue] (13,3) node[anchor=south west] {$\bm{m}$};

\end{tikzpicture} 
\caption{The following picture displays the reversed sequences in Figure \ref{type I fig}. These sequences are natural for Theorem \ref{thm:type II thm} in the case $r = 2$ with $\deg{\Phi} = \deg{\Psi_1} = \deg{\Psi_1} = 2$, and the choice $\bm{m} = (n_1+1,n_2+1)$. 
}
\end{figure}

\begin{proof}
    We first prove the result in the case $\bm{n} \geq \bm{M}$. {Note that the right-hand side of~\eqref{eq:det formula rational type II} is a polynomial. Denote it by $\widetilde{P}_{\bm{n}}$, and let us show that this polynomial is indeed the type II multiple orthogonal polynomial for $\widetilde{\bm{\mu}}$.} 
    Since each index in the sequence $\set{\bm{m}-\bm{s}_j}_{j = 0}^{N + \abs{\bm{M}}}$ is below $\bm{m}$, we have $\deg{\widetilde{P}_{\bm{n}}} \leq \abs{\bm{n}}$. Since the sequence is above $\bm{n}-\bm{M}$, we have, by $\Psi_j\widetilde{\mu}_j = \Phi \mu_j$ and \eqref{eq:det formula rational type II}, 
    \m{\widetilde{\mu}_j[\widetilde{P}_{\bm{n}}(x)\Psi_j(x)x^p] = \mu_j[\Phi(x)\widetilde{P}_{\bm{n}}(x)x^p] = 0, \qquad p = 0,\dots,n_j-M_j-1.}
    If $p < M_j$, consider the partial fraction decomposition
    \nm{eq:type II partial fraction}{x^p = \sum_{k = 1}^{M_j} \alpha_{p,j,k} \frac{\Psi_j(x)}{x-w_{j,k}}.}
    From this we get
    \m{\widetilde{\mu}_j[\widetilde{P}_{\bm{n}}(x)x^p] = \sum_{k=1}^{M_j}\alpha_{p,j,k}\widetilde{\mu}_j\left[\widetilde{P}_{\bm{n}}(x)\frac{\Psi_j(x)}{x-w_{j,k}}\right] = \sum_{k=1}^{M_j}\alpha_{p,j,k}\widecheck{\mu}_j\left[\Phi(x)\widetilde{P}_{\bm{n}}(x)\frac{\Psi_j(x)}{x-w_{j,k}}\right] = 0,}
    since we get the determinant in \eqref{eq:det formula rational type II} but with the first row replaced by a linear combination of the last $\abs{\bm{M}}$ rows. Since every polynomial of degree $\leq n_j - 1$ is in the span of $\set{x^p}_{p=0}^{M-1}\cup\set{\Psi_j(x)x^p}_{p=0}^{n_j-M-1}$, we end up with
    \m{\widetilde{\mu}_j[\widetilde{P}_{\bm{n}}(x)x^p] = 0, \qquad p = 0,\dots,n_j-1, \qquad j = 1,\dots,r.}
    
    The sequence starts at $\bm{m}$ and decreases, so the degree $\abs{\bm{n}}$ coefficient of $\widetilde{P}_{\bm{n}}$ is equal to $D^{(II)}_{\bm{n}}$. We have $D^{(II)}_{\bm{n}} = 0$ if and only if there are constants $c_1,\dots,c_{N+\abs{\bm{M}}}$, not all zero, such that 
    \begin{align}
        \label{eq:zerosDivisibility2}
        & \sum_{i = 1}^{N+\abs{\bm{M}}}c_i P_{\bm{m}-\bm{s}_{i}}(x) = 0, \qquad x = z_1,\dots,z_N, \\\label{eq:QzerosDivisibility2}
& \sum_{i = 1}^{N+\abs{\bm{M}}} c_i Q_{\bm{m}-\bm{s}_{i}}^{(j)}(x) = 0, \qquad x = w_{j,1},\dots,w_{j,M_j}, \qquad j = 1,\dots,r.
\end{align}
{Now suppose $D^{(II)}_{\bm{n}} = 0$ and choose $c_1,\dots,c_{N+\abs{\bm{M}}}$ as above. By~\eqref{eq:zerosDivisibility2}, we know that} 
\m{P(x) = \Phi(x)^{-1}\sum_{i = 1}^{N+\abs{\bm{M}}} c_i P_{\bm{m}-\bm{s}_{i}}(x).}
{is a polynomial.}
Since our sequence of indices is admissible, Lemma~\ref{lem:linear independence modified paths type II} implies $P \not\equiv 0$. Since every element is strictly below $\bm{m}$ we have $\deg{P} < \abs{\bm{n}}$, and since every element is above $\bm{n}-\bm{M}$ we have, by similar arguments as earlier,
\begin{equation}\label{eq:orthogonality relations P}
    {\widetilde{\mu}_j[P(x)x^p] = 0, \qquad p = 0,\dots,n_j-1, \qquad j = 1,\dots,r.}
\end{equation}
Hence $P$ is a type II polynomial for the index $\bm{n}$ with respect to $\widetilde{\bm{\mu}}$, so $\bm{n}$ cannot be normal for $\widetilde{\bm{\mu}}$. 

Conversely, suppose $\bm{n}$ is not normal for $\widetilde{\bm{\mu}}$. Then there some polynomial $P \not\equiv 0$ with $\deg{P} < \abs{\bm{n}}$ such that \eqref{eq:orthogonality relations P} holds. As we have seen, this implies
\begin{equation*}
    \mu_j[\Phi(x)P(x)x^p] = 0, \qquad p = 0,\dots,n_j-M_j-1, \qquad j = 1,\dots,r.
\end{equation*}
{Note that $\deg (\Phi(x)P(x)) < N+|\bm{n}| = |\bm{m}|$.} For a polynomial $R(x) = \sum_{k = 0}^{\abs{\bm{m}}-1}\kappa_kx^k$, the orthogonality relations 
\begin{equation}\label{eq:lastEq}
    \mu_j[R(x)x^p] = 0, \qquad p = 0,\dots,n_j-M_j-1, \qquad j = 1,\dots,r,
\end{equation}
turn into a $(\abs{\bm{n}}-\abs{\bm{M}})\times\abs{\bm{m}}$ system of equations. The coefficient matrix of this system is given by $H_{\bm{m}}$ with $N+\abs{\bm{M}}$ rows removed. Hence the solution space of this system has dimension $N+\abs{\bm{M}}$. Note that {$\bm{m}-\bm{s}_i \ge \bm{n}-\bm{M}$, and therefore} $P_{\bm{m}-\bm{s}_i}$ is a solution {of~\eqref{eq:lastEq}} for $i = 1,\dots,N+\abs{\bm{M}}$. Since these polynomials are linearly independent by Lemma \ref{lem:linear independence modified paths type II}, we get $\Phi(x)P(x) = \sum_{i = 1}^{N+\abs{\bm{M}}} c_i P_{\bm{m}-\bm{s}_{i}}(x)$ for some constants $c_i$, which are clearly not all zero. 
{This implies \eqref{eq:zerosDivisibility2} since $\Phi(z_j)=0$. Equalities~\eqref{eq:QzerosDivisibility2} also hold as 
$$
\sum_{i = 1}^{N+\abs{\bm{M}}} c_i Q_{\bm{m}-\bm{s}_{i}}^{(j)}(w_{j,k}) = \widecheck{\mu}_j \Big[ \Phi(x) P(x) \frac{\Psi_j(x)}{x-w_{j,k}} \Big]
=
\widetilde{\mu}_j \Big[ P(x) \frac{\Psi_j(x)}{x-w_{j,k}} \Big]=0,
$$
by \eqref{eq:orthogonality relations P} and $M_j-1\le n_j-1$. 
Finally, \eqref{eq:zerosDivisibility2}-\eqref{eq:QzerosDivisibility2} implies $D^{(II)}_{\bm{n}} = 0$.}

Now consider the case $n_j < M_j$ for some $j$. We want to check orthogonality conditions with respect to $\widetilde{\mu}_j$~\eqref{eq:orthogonality relations P} for $p \leq n_j - 1$ (note that if $n_k \geq M_k$ for some $k \neq j$ then the orthogonality conditions for $\widetilde{\mu}_k$ work out the same as before). As in the proof of the type I formula, we substitute \eqref{eq:identity} into the partial fraction decomposition \eqref{eq:type II partial fraction} to get 
\nm{eq:magic type II}{\sum_{k = 1}^{M_j}\alpha_{p,j,k} w_{j,k}^q = 0, \qquad q = 0,\dots,M_j-p-2.}
Since $p < M_j$, we have
\m{\widetilde{\mu}_j[\widetilde{P}_{\bm{n}}(x)x^p] = \sum_{k=1}^{M_j}\alpha_{p,j,k}\widecheck{\mu}_j\left[\Phi(x)\widetilde{P}_{\bm{n}}(x)\frac{\Psi_j(x)}{x-w_{j,k}}\right],}
which is the determinant in \eqref{eq:det formula rational type II} but with each of the first $N + \abs{\bm{M}^*}+1$ elements in the first row replaced with a linear combination of the elements in the rows corresponding to $Q^{(j)}$'s. Use this to get the first $N + \abs{\bm{M}^*}+1$ elements in the first row of the determinant equal to $0$ through elementary row operations. Then a linear combination of the rows in the $j$-th block of \eqref{eq:q type II small indices} is also added to the first row. This linear combination is $0$ by \eqref{eq:magic type II}, since $M_j - p - 2 \geq M_j - n_j - 1$, so all the elements in the first row of the determinant are now $0$. Hence we get
\m{\widetilde{\mu}_j[\widetilde{P}_{\bm{n}}(x)x^p] = 0, \qquad p = 0,\dots,n_j-1.}
Similar arguments then show that $D^{(II)}_{\bm{n}} \neq 0$ if and only if $\bm{n}$ is normal for $\widetilde{\bm{\mu}}$. 
\end{proof}

\subsection{The Case of Zeros of Higher Multiplicity}\label{general case}\hfill\\

Some modifications of Theorem \ref{thm:type I thm} and Theorem \ref{thm:type II thm} can give us a formula for any (diagonal) rational perturbation of functionals, i.e., $\widetilde{\bm{\mu}} = (\widetilde{\mu}_1,\dots,\widetilde{\mu}_r)$ satisfying 
\nm{eq:rp functional moprl}{\Psi_j\widetilde{\mu}_j = \Phi_j\mu_j, \qquad j = 1,\dots,r,}
for any polynomials $\Phi_j(x) = \prod_{k = 1}^{N_j}(x-z_{j,k})$ and $\Psi_j(x) = \prod_{k = 1}^{M_j}(x-w_{j,k})$, $j = 1,\dots,r$. For the case of measures this includes all perturbations $\widetilde{\mu}_j$ of the type 
\nm{eq:rp measure moprl}{\int f(x)d\widetilde{\mu}_j(x) = \int f(x)\frac{\Phi_j(x)}{\Psi_j(x)}d\mu_j(x) + \sum_{k = 1}^{K_j}c_{j,k}f(u_{j,k}), \qquad j = 1,\dots,r,}
since with $\Lambda_j(x) = \prod_{k = 1}^{K_j}(x-u_{j,k})$ we have 
\nm{eq:rp w discrete parts}{\Lambda_j\Psi_j\widetilde{\mu}_j = \Lambda_j\Phi_j\mu_j, \qquad j = 1,\dots,r.}

First let $\Psi = \lcm(\Psi_1,\dots,\Psi_r)$ and $\Phi_j^* = \Phi_j\Psi/\Psi_j$, $j = 1,\dots,r$, and note that we have
\nm{eq:modified rp functional type I}{\Psi\widetilde{\mu}_j = \Phi_j^*\mu_j, \qquad j = 1,\dots,r.}
From here we can apply Theorem \ref{thm:type I thm}, as long as all $\Phi_1,\dots,\Phi_r$ and $\Psi_1,\dots,\Phi_r$ have simple roots, and $\Phi_i$ and $\Psi_j$ are coprime for $i \neq j$. Similarly, we can apply Theorem \ref{thm:type II thm} using 
\nm{eq:modified rp functional type II}{\Psi_j^*\widetilde{\mu}_j = \Phi\mu_j, \qquad j = 1,\dots,r,}
where $\Phi = \lcm(\Phi_1,\dots,\Phi_r)$ and $\Psi_j^* = \Psi_j\Phi/\Phi_j$, $j = 1,\dots,r$. 

We now describe how to modify Theorem \ref{thm:type I thm} to allow multiplicity of the roots of $\Phi_1,\dots,\Phi_r$ and $\Psi$. The same ideas can then be used to modify Theorem \ref{thm:type II thm} to the case of higher multiplicities. Using \eqref{eq:modified rp functional type I} and \eqref{eq:modified rp functional type II} we would then get determinantal formulas for type I and type II polynomials, for any choice of $\Phi_1,\dots,\Phi_r$ and $\Psi_1,\dots,\Psi_r$ in \eqref{eq:rp functional moprl}. 

Recall that the zeros of $\Phi_j$ are $\{z_{j,k}\}_{k=1}^{N_j}$, repeated according to their multiplicities. Denote $l_{j,k}$ 
to be the number of $i < k$ such that $z_{j,i} = z_{j,k}$, 
and modify \eqref{eq:type I blocks} by
\nm{eq:type I block modified}{A_{\bm{n}}^{(j)}(\bm{z}_j) = \begin{pmatrix}
    \frac{d^{l_{j,1}}}{d z^{l_{j,1}}}A_{\bm{n}}^{(j)}(z){\Big|}_{z = z_{j,1}} \\
    \vdots \\
    \frac{d^{l_{j,N_j}}}{d z^{l_{j,N_j}}}A_{\bm{n}}^{(j)}(z){\Big|}_{z = z_{j,N_j}}
\end{pmatrix}, \qquad j = 1,\dots,r, \qquad \bm{n} \in \N^r.}
This ensures that the determinant in \eqref{eq:det formula rational type I} remains divisible by $\Phi_j$. From there the rest of the proof works out without any changes (and all of this includes the special case $\abs{\bm{n}} \leq M$). 

Similarly, 
write $l_k$ for the number of $i < k$ such that $w_i = w_k$. Here we modify \eqref{eq:q type I} by
\nm{eq:q type I modified}{B_{\bm{n}}(w_k) = \sum_{j = 1}^{r}\widecheck{\mu}_j\left[A_{\bm{n}}^{(j)}(x)\frac{\d^{l_k}}{\d z^{l_k}}\frac{\Psi(x)}{x-z}\right]{\Bigg|}_{z = w_k} = \sum_{j = 1}^{r}\widecheck{\mu}_j\left[A_{\bm{n}}^{(j)}(x)\frac{l_k!\Psi(x)}{(x-w_k)^{l_k+1}}\right], \qquad k = 1,\dots,N, \\ \bm{n} \in \N^r.}
When $\bm{m}+\bm{s}_i \notin \N^r\setminus\set{\bm{0}}$ in Theorem \ref{thm:type I thm}, we put $B_{\bm{m}+\bm{s}_i}(w_k) = \frac{d^{l_k}}{dz^{l_k}}z^{M-\abs{\bm{n}}-i}|_{z = w_k}$. For the proof, the orthogonality with respect to $\Psi(x)x^p$ remains the same, but for the orthogonality with respect to $x^p$ for $p < M$, we need to replace the partial fraction decomposition \eqref{eq:type I partial fraction} with 
\nm{eq:type I modified partial fraction}{x^p = \sum_{k = 1}^{M}\alpha_{p,k}\frac{l_k!\Psi(x)}{(x-w_k)^{l_k+1}}.}
By differentiating \eqref{eq:identity} in $z$ we obtain
\nm{eq:modified identity}{\frac{l!x^q}{(x-z)^{l+1}} = x^{q-1}\frac{d^l}{dz^l}1 + x^{q-2}\frac{d^l}{dz^l}z + \dots + \frac{d^l}{dz^l}z^{q-1} + o(1) \qquad (x \rightarrow \infty).}
Substitute this into \eqref{eq:type I modified partial fraction} multiplied with $x^q$ to get
\nm{eq:magic modified type I}{\sum_{k = 1}^{M}\alpha_{p,k} \frac{d^{l_k}}{dz^{l_k}}z^q{\Big|}_{w_k} = \begin{cases}
    0, \qquad q = 0,\dots,M-p-2, \\
    1, \qquad q = M-p-1.
\end{cases}}
From here the rest of the proof of Theorem \ref{thm:type I thm} works out the same way as before.

\section{Special cases}\label{examples}

\subsection{Christoffel Transforms}\hfill\\

From now on we assume all systems we consider are perfect. We ignore the normalizing constants in most of our determinantal formulas, but since we assume perfectness, Theorem \ref{thm:type I thm} and Theorem \ref{thm:type II thm} ensure that our determinants are multiple orthogonal polynomials.

The most general Christoffel transforms of the system $\bm{\mu} = (\mu_1,\dots,\mu_r)$ that we consider in this paper will be of the form 
\nm{eq:ct example moprl}{\widehat{\bm{\mu}} = (\widehat{\mu}_1,\dots,\widehat{\mu}_r) = (\Phi_1\mu_1,\dots,\Phi_r\mu_r),}
for any choice of polynomials $\Phi_j(x) = \prod_{k = 1}^{N_j}(x-z_{j,k})$, $j = 1,\dots,r$. However, we will assume $\Phi_1,\dots,\Phi_r$ all have only simple zeros, to avoid having to add rows of derivatives as described in Section \ref{general case}. We write $\widehat{\bm{A}}_{\bm{n}} = (\widehat{A}_{\bm{n}}^{(1)},\dots,\widehat{A}_{\bm{n}}^{(r)})$ for the type I polynomials with respect to $\widehat{\bm{\mu}}$, $\bm{A}_{\bm{n}} = (A_{\bm{n}}^{(1)},\dots,A_{\bm{n}}^{(r)})$ for the type I polynomials with respect to $\bm{\mu}$, $\widehat{P}_{\bm{n}}$ for the type II polynomials with respect to $\widehat{\bm{\mu}}$, and $P_{\bm{n}}$ for the type II polynomials with respect to $\bm{\mu}$. We also write $\bm{N}$ for the index $(N_1,\dots,N_r)$. 


\subsubsection{\textnormal{\textbf{Systems $\widehat{\bm{\mu}} = (\Phi_1\mu_1,\Phi_2\mu_2,\dots,\Phi_r\mu_r)$, type I polynomials} }}\hfill

\textnormal{For type I polynomials, Theorem \ref{thm:type I thm} applies immediately with $\Psi_1 = \cdots = \Psi_r = 1$, reducing \eqref{eq:det formula rational type I} to 
\nm{eq:general ct type I}{\widehat{A}_{\bm{n}}^{(j)}(x) = \Phi_j(x)^{-1}\det{\begin{pmatrix}
A_{\bm{n}}^{(j)}(x) & A_{\bm{n}+\bm{s}_{1}}^{(j)}(x) & \cdots & A_{\bm{n}+\bm{s}_{\abs{\bm{N}}}}^{(j)}(x) \\
A_{\bm{n}}^{(1)}(\bm{z}_1) & A_{\bm{n}+\bm{s}_{1}}^{(1)}(\bm{z}_1) & \cdots & A_{\bm{n}+\bm{s}_{\abs{\bm{N}}}}^{(1)}(\bm{z}_1) \\
\vdots & \vdots & \ddots & \vdots \\
A_{\bm{n}}^{(r)}(\bm{z}_{r}) & A_{\bm{n}+\bm{s}_{1}}^{(r)}(\bm{z}_r) & \cdots & A_{\bm{n}+\bm{s}_{\abs{\bm{N}}}}^{(r)}(\bm{z}_r) \\
\end{pmatrix}}, \qquad j = 1\dots,r.}
Here we chose $\bm{m} = \bm{n}$, but note that Theorem \ref{thm:type I thm} allows for any $\bm{m} \leq \bm{n} + \bm{N}$ with $\abs{\bm{m}} = \abs{\bm{n}}$. For the sequence $\set{\bm{n}+\bm{s}_j}_{j = 0}^{\abs{\bm{N}}}$, the natural choice may be a path from $\bm{n}$ to $\bm{n}+\bm{N}$, or perhaps 
the increasing frame from $\bm{n}$ towards $\bm{n}+\bm{N}$ (see Definition~\ref{def:forward sequences}).}

\subsubsection{\textnormal{\textbf{Systems $\widehat{\bm{\mu}} = (\Phi\mu_1,\Phi\mu_2,\dots,\Phi\mu_r)$, type II polynomials}}}\label{ex:typeIIFullChr} \hfill

\textnormal{For type II polynomials, Theorem \ref{thm:type II thm} does not apply immediately in the general case, since it requires $\Phi_1 = \cdots = \Phi_r$. We can still use the trick \eqref{eq:modified rp functional type II} to get a determinantal formula, and since $\Psi_1 = \cdots = \Psi_r = 1$ we do not have to deal with any zeros of higher multiplicity.}

\textnormal{If $\Phi_1,\dots,\Phi_r$ are coprime we use $\Phi(x) = \Phi_1(x)\cdots \Phi_r(x)$ (of degree $|\bm{N}|$) and $\Psi_j(x) = \Phi(x)/\Phi_j(x)$ (of degree $|\bm{N}|-N_j$) in Theorem \ref{thm:type II thm}. The determinant~\eqref{eq:det formula rational type II} is then of size $(r\abs{\bm{N}}+1)\times(r\abs{\bm{N}}+1)$. Note that this is much larger than the $(\abs{\bm{N}}+1)\times(\abs{\bm{N}}+1)$ determinant in \eqref{eq:general ct type I}. 
}

\textnormal{The more zeros $\Phi_1,\dots,\Phi_r$ have in common the smaller the determinant in the type II formula will be. In particular the simplest case here is $\Phi_1(x) = \cdots = \Phi_r(x) = \Phi(x) = \prod_{k = 1}^{N}(x-z_k)$. This was the setting we considered in \cite{ctpaper}. Here Theorem \ref{thm:type II thm} even applies directly, reducing \eqref{eq:det formula rational type II} to
\nm{eq:full ct type II}{\widehat{P}_{\bm{n}}(x) = \Phi(x)^{-1}\det{\begin{pmatrix}
P_{\bm{n}+N\bm{e}_1}(x) & P_{\bm{n}+(N-1)\bm{e}_1}(x) & \cdots & P_{\bm{n}}(x) \\
P_{\bm{n}+N\bm{e}_1}(z_1) & P_{\bm{n}+(N-1)\bm{e}_1}(z_1) & \cdots & P_{\bm{n}}(z_1) \\
\vdots & \vdots & \ddots & \vdots \\
P_{\bm{n}+N\bm{e}_1}(z_N) & P_{\bm{n}+(N-1)\bm{e}_1}(z_N) & \cdots & P_{\bm{n}}(z_N)
\end{pmatrix}}.}
Note the similarity with \eqref{eq:ct thm one measure}. We chose the sequence to be the straight path from $\bm{m} = \bm{n} + N\bm{e}_1$ to $\bm{n}$. Of course a straight path from $\bm{m} = \bm{n} + N\bm{e}_j$ to $\bm{n}$ works for any $j$, or more generally, any path from $\bm{m} \geq \bm{n}$ to $\bm{n}$ with $\abs{\bm{m}} = \abs{\bm{n}} + N$. 
}

\subsubsection{\textnormal{\textbf{Systems $\widehat{\bm{\mu}} = (\Phi\mu_1,\mu_2,\dots,\mu_r)$, type I polynomials}}}\label{ex:partial ct type I} \hfill

\textnormal{The determinant in \eqref{eq:full ct type II} is of size $(N+1)\times(N+1)$, which is even smaller than the determinant in \eqref{eq:general ct type I}. Also note that even for that special case, the determinant in \eqref{eq:general ct type I} remains at the same size $(\abs{\bm{N}}+1)\times(\abs{\bm{N}}+1) = (rN+1)\times(rN+1)$. We now consider Christoffel transforms of the form $(\Phi\mu_1,\mu_2,\dots,\mu_r)$. This system satisfies dual properties to the system $(\Phi\mu_1,\Phi_2\mu_2,\dots,\Phi_r\mu_r)$.  In particular this system reduces the determinant in \eqref{eq:general ct type I} to size $(N+1)\times(N+1)$. Indeed, \eqref{eq:general ct type I} reduces to 
\nm{eq:partial ct type I}{& \widehat{A}_{\bm{n}}^{(1)}(x) = \Phi(x)^{-1}\det{\begin{pmatrix}
A_{\bm{n}}^{(1)}(x) & A_{\bm{n}+\bm{e}_{1}}^{(1)}(x) & \cdots & A_{\bm{n}+N\bm{e}_{1}}^{(1)}(x) \\
A_{\bm{n}}^{(1)}(z_1) & A_{\bm{n}+\bm{e}_{1}}^{(1)}(z_1) & \cdots & A_{\bm{n}+N\bm{e}_{1}}^{(1)}(z_1) \\
\vdots & \vdots & \ddots & \vdots \\
A_{\bm{n}}^{(1)}(z_N) & A_{\bm{n}+\bm{e}_{1}}^{(1)}(z_N) & \cdots & A_{\bm{n}+N\bm{e}_{1}}^{(1)}(z_N) \\
\end{pmatrix}}, \\
& \widehat{A}_{\bm{n}}^{(j)}(x) = \det{\begin{pmatrix}
A_{\bm{n}}^{(j)}(x) & A_{\bm{n}+\bm{e}_{1}}^{(j)}(x) & \cdots & A_{\bm{n}+N\bm{e}_{1}}^{(j)}(x) \\
A_{\bm{n}}^{(1)}(z_1) & A_{\bm{n}+\bm{e}_{1}}^{(1)}(z_1) & \cdots & A_{\bm{n}+N\bm{e}_{1}}^{(1)}(z_1) \\
\vdots & \vdots & \ddots & \vdots \\
A_{\bm{n}}^{(1)}(z_N) & A_{\bm{n}+\bm{e}_{1}}^{(1)}(z_N) & \cdots & A_{\bm{n}+N\bm{e}_{1}}^{(1)}(z_N) \\
\end{pmatrix}}, \qquad j = 2,\dots,r.}
Note that with our choice $\bm{m} = \bm{n}$, the sequence of multi-indices $\{\bm{n}+j\bm{e}_1\}_{j=0}^N$ we used here is the only possible choice. }

\subsubsection{\textnormal{\textbf{Systems $\widehat{\bm{\mu}} = (\Phi\mu_1,\mu_2,\dots,\mu_r)$, type II polynomials}}}\label{ex:partial ct type II} \hfill

\textnormal{For the systems considered in the previous example, we can also get a type II formula, but with a determinant of larger size than in the type I formula. Similarly to Section~\ref{ex:typeIIFullChr},  Theorem \ref{thm:type I thm} does not apply directly to this system for type II polynomials, so we again use the trick \eqref{eq:modified rp functional type II}. 
 We note that 
\nm{eq:type II parial christoffel property}{\widehat{\mu}_1 = \Phi {\mu}_1, \qquad \Phi\widehat{\mu}_j = \Phi\mu_j, \qquad j = 2,\dots,r.}
Hence we apply Theorem \ref{thm:type II thm} with $\Phi_1 = \cdots = \Phi_r = \Phi$, $\Psi_1 = 1$, $\Psi_2 = \cdots = \Psi_r = \Phi$. We obtain the determinantal formula
\nm{eq:partial ct type II}{\widehat{P}_{\bm{n}-N\bm{e}_1}(x) & = \Phi(x)^{-1}\det{\begin{pmatrix}
        P_{\bm{n}}(x) & P_{\bm{n}-\bm{s}_1}(x) & \dots & P_{\bm{n}-\bm{s}_{rN}}(x) \\
        P_{\bm{n}}(\bm{z}) & P_{\bm{n}-\bm{s}_1}(\bm{z}) & \dots & P_{\bm{n}-\bm{s}_{rN}}(\bm{z}) \\
        Q_{\bm{n}}^{(2)}(\bm{z}) & Q_{\bm{n}-\bm{s}_1}^{(2)}(\bm{z}) & \dots & Q_{\bm{n}-\bm{s}_{rN}}^{(2)}(\bm{z}) \\
        \vdots & \vdots & \ddots & \vdots \\
        Q_{\bm{n}}^{(r)}(\bm{z}) & Q_{\bm{n}-\bm{s}_1}^{(r)}(\bm{z}) & \dots & Q_{\bm{n}-\bm{s}_{rN}}^{(r)}(\bm{z}) \\
    \end{pmatrix}}, \qquad n_1 \geq N.}
To simplify the formula we chose the index $\bm{n}-N\bm{e}_1$ on the left-hand side instead of $\bm{n}$. This allows the sequence of indices in the determinant to start at $\bm{n}$, and $\set{\bm{n}-\bm{s}_j}_{j = 0}^{rN}$ is then any admissible sequence from $\bm{n} = (n_1,\dots,n_r)$ towards $\bm{n} - N\bm{1} = (n_1-N,\dots,n_r-N)$. 
Here we assume that the multi-index $\bm{n} - N\bm{1}$ stays in $\N^r$. For the special case when the sequence $\{\bm{n}-\bm{s}_j\}_{j=0}^{rN}$ leaves $\N^r$, see Theorem \ref{thm:type II thm}).
}

\textnormal{
The determinant is of the larger size $(rN+1)\times(rN+1)$, so we are using the compact notation \eqref{eq:type II blocks}. Here the $Q$'s are given by
\nm{eq:q partial ct type II}{& Q_{\bm{n}}^{(j)}(z) = \widecheck{\mu}_j\left[P_{\bm{n}}(x)\frac{\Phi(x)}{x-z}\right], \qquad z = z_1,\dots,z_N, \\
& Q_{\bm{n}}^{(j)}(\bm{z}) = \begin{pmatrix}
    Q_{\bm{n}}^{(j)}(z_1) \\
    \vdots \\
    Q_{\bm{n}}^{(j)}(z_N)
\end{pmatrix} \qquad j = 2,\dots,r, \qquad \bm{n} \in \N^r,}
where $\widecheck{\mu}_j$ satisfies $\Phi\widecheck{\mu}_j = \mu_j$, $j = 2,\dots,r$.
In the case of measures, with all zeros outside the support of each $\mu_2,\dots,\mu_r$, the choice $d\widecheck{\mu}_j(x) = \Phi(x)^{-1}d\mu_j(x)$ turns \eqref{eq:q partial ct type II} into the Cauchy transforms
\nm{eq:q partial ct type II measures}{Q_{\bm{n}}^{(j)}(z) = \int\frac{P_{\bm{n}}(x)}{x-z}d\mu_j(x), \qquad j = 2,\dots,r, \qquad \bm{n} \in \N^r.}
} 

\textnormal{The determinantal formula \eqref{eq:partial ct type II} was one of the main motivations for the more general Theorem \ref{thm:type I thm} and Theorem \ref{thm:type II thm}, as $Q$'s enter into the formula, similarly to Uvarov's formula \eqref{eq:rp thm one measure}. 
}

\subsection{Geronimus Transforms}\label{ss:generalGeronimus}\hfill\\

\textnormal{The most general Geronimus transforms we consider in this paper are systems $\widecheck{\bm{\mu}} = (\widecheck{\mu}_1,\dots,\widecheck{\mu}_r)$ satisfying 
\nm{eq:gt property moprl}{(\Psi_1\widecheck{\mu}_1,\dots,\Psi_r\widecheck{\mu}_r) = (\mu_1,\dots,\mu_r)} 
for any choice of polynomials $\Psi_j(x) = \prod_{k = 1}^{M_j}(x-w_{j,k})$, $j = 1,\dots,r$. We assume that $\Psi_1,\dots,\Psi_r$ all have simple zeros. When $\mu_j$ is a measure and $\Psi_j$ has no zeros in $\supp(\mu_j)$, $\widecheck{\mu}_j$ is then a measure given by
\nm{eq:gt moprl measures}{\int f(x)d\widecheck{\mu}_j(x) = \int \frac{f(x)}{\Psi_j(x)}d\mu_j(x) + \sum_{k = 1}^{M_j}c_{j,k}f(w_{j,k}).}
We write $\widecheck{\bm{A}}_{\bm{n}} = (\widecheck{A}_{\bm{n}}^{(1)},\dots,\widecheck{A}_{\bm{n}}^{(r)})$ for the type I polynomials with respect to $\widecheck{\bm{\mu}}$, $\bm{A}_{\bm{n}} = (A_{\bm{n}}^{(1)},\dots,A_{\bm{n}}^{(r)})$ for the type I polynomials with respect to $\bm{\mu}$, $\widecheck{P}_{\bm{n}}$ for the type II polynomials with respect to $\widecheck{\bm{\mu}}$, and $P_{\bm{n}}$ for the type II polynomials with respect to $\bm{\mu}$. We also write $\bm{M}$ for the index $(M_1,\dots,M_r)$. }

\subsubsection{\textnormal{\textbf{Systems $\widecheck{\bm{\mu}}=(\Psi_1^{-1}\mu_1,\Psi_2^{-1}\mu_2,\dots,\Psi_r^{-1}\mu_r)$, type II polynomials}}}\label{ex:GeneralGeronimusTypeII}  \hfill

\textnormal{The determinantal formulas presented in this section further illustrate the duality between the type I and type II setting, as well as the duality between multiplying all measures by a factor and multiplying only one of the measures by the same factor. We first see this by noting that the type II formula is better adapted to the general Geronimus transforms than the type I formula. Here we get
\nm{eq:general gt type II}{\widecheck{P}_{\bm{n}}(x) = \det{\begin{pmatrix}
P_{\bm{n}}(x) & P_{\bm{n}-\bm{s}_{1}}(x) & \cdots & P_{\bm{n}-\bm{s}_{\abs{\bm{M}}}}(x) \\
Q_{\bm{n}}^{(1)}(\bm{w}_1) & Q_{\bm{n}-\bm{s}_{1}}^{(1)}(\bm{w}_1) & \cdots & Q_{\bm{n}-\bm{s}_{\abs{\bm{M}}}}^{(1)}(\bm{w}_1) \\
\vdots & \vdots & \ddots & \vdots \\
Q_{\bm{n}}^{(r)}(\bm{w}_r) & Q_{\bm{n}-\bm{s}_{1}}^{(r)}(\bm{w}_r) & \cdots & Q_{\bm{n}-\bm{s}_{\abs{\bm{M}}}}^{(r)}(\bm{w}_r) \\
\end{pmatrix}},}
where $\set{\bm{n}+\bm{s}_j}_{j = 0}^{\abs{\bm{M}}}$ is an admissible sequence from $\bm{n}$ towards $\bm{n}-\bm{M}$, including paths from $\bm{n}$ to $\bm{n}-\bm{M}$, and the frame from $\bm{n}$ towards $\bm{n}-\bm{M}$. 
$Q_{\bm{n}}^{(j)}(\bm{w}_j)$ in \eqref{eq:general gt type II} is given by 
\nm{eq:q gt type II}{& Q_{\bm{n}}^{(j)}(w_{j,k}) = \widecheck{\mu}_j\left[P_{\bm{n}}(x)\frac{\Psi_j(x)}{x-w_{j,k}}\right], \qquad k = 1,\dots,M_j, \\
& Q_{\bm{n}}^{(j)}(\bm{w}_j) = \begin{pmatrix}
    Q_{\bm{n}}^{(j)}(w_{j,1}) \\
    \vdots \\
    Q_{\bm{n}}^{(j)}(w_{j,M_j})
\end{pmatrix}, \qquad j = 1,\dots,r, \qquad \bm{n} \in \N^r.}
When $\mu_j$ is the measure \eqref{eq:gt moprl measures} and $\Psi_j$ has no zeros in $\supp(\mu_j)$, we have
\nm{eq:q gt type II measures}{Q_{\bm{n}}^{(j)}(w_{j,k}) = \int \frac{P_{\bm{n}}(x)}{x-w_{j,k}}d\mu_j(x) + c_{j,k}P_{\bm{n}}(w_{j,k})\Psi_j'(w_{j,k}), \qquad k = 1,\dots,M_j.}
In particular, $Q_{\bm{n}}$ is a Cauchy transform when we have no discrete part. 
}

\subsubsection{\textnormal{\textbf{Systems $\widecheck{\bm{\mu}}=(\Psi^{-1}\mu_1,\Psi^{-1}\mu_2,\dots,\Psi^{-1}\mu_r)$, type I polynomials}}}  \hfill

\textnormal{The case $\Phi_1(x) = \cdots = \Phi_r(x) = \Phi(x) = \prod_{k = 1}^N(x-z_k)$ simplified the determinantal formula for the Christoffel transform of type II polynomials to an $(N+1)\times(N+1)$ determinant. For the Geronimus transform, this now happens to type I polynomials when~$\Psi_1(x) = \cdots = \Psi_r(x) = \Psi(x) = \prod_{k = 1}^M (x-w_k)$, as \eqref{eq:det formula rational type I} reduces to
\nm{eq:full gt type I}{\widecheck{A}_{\bm{n}}^{(j)}(x) = \det{\begin{pmatrix}
A_{\bm{n}-M\bm{e}_1}^{(j)}(x) & A_{\bm{n}-(M-1)\bm{e}_1}^{(j)}(x) & \cdots & A_{\bm{n}}^{(j)}(x) \\
B_{\bm{n}-M\bm{e}_1}(w_1) & B_{\bm{n}-(M-1)\bm{e}_1}(w_1) & \cdots & B_{\bm{n}}(w_1) \\
\vdots & \vdots & \ddots & \vdots \\
B_{\bm{n}-M\bm{e}_1}(w_M) & B_{\bm{n}-(M-1)\bm{e}_1}(w_M) & \cdots & B_{\bm{n}}(w_M)
\end{pmatrix}}, \qquad j = 1\dots,r.}
Here $B_{\bm{n}}(w)$ is given by
\nm{eq:q full gt type I}{B_{\bm{n}}(w_k) = \sum_{j = 1}^r\widecheck{\mu}_j\left[A_{\bm{n}}^{(j)}(x)\frac{\Psi(x)}{x-w_k}\right], \qquad k = 1,\dots,M, \qquad \bm{n}\in\N^r.}
When $\widecheck{\mu}_j$ is given by~\eqref{eq:gt moprl measures} for each $j = 1,\dots,r$ we have
\nm{eq:q full gt type I measures}{B_{\bm{n}}(w_k) = \sum_{j = 1}^r\left(\int\frac{A_{\bm{n}}^{(j)}(x)}{x-w_k}d\mu_j(x) + c_{j,k}A_{\bm{n}}^{(j)}(w_k)\Psi_j'(w_k)\right), \qquad k = 1,\dots,M.}}

\subsubsection{\textnormal{\textbf{Systems $\widecheck{\bm{\mu}}=(\Psi^{-1}\mu_1,\mu_2,\dots,\mu_r)$, type II and type I polynomials}}}  \hfill

\textnormal{The determinant in \eqref{eq:full gt type I} is of size $(M+1)\times(M+1)$, but for type I polynomials the determinant is of size $(rM+1)\times(rM+1)$. To get an $(M+1)\times(M+1)$ determinant for type II polynomials we instead consider the case $\Psi_1 = \Psi$ and $\Psi_2 = \cdots = \Psi_r = 1$ in \eqref{eq:gt property moprl}. Then \eqref{eq:general gt type II} reduces to
\nm{eq:partial gt type II}{\widecheck{P}_{\bm{n}}(x) = \det{\begin{pmatrix}
P_{\bm{n}}(x) & P_{\bm{n}-\bm{e}_{1}}(x) & \cdots & P_{\bm{n}-M\bm{e}_{1}}(x) \\
Q_{\bm{n}}^{(1)}(w_1) & Q_{\bm{n}-\bm{e}_{1}}^{(1)}(w_1) & \cdots & Q_{\bm{n}-M\bm{e}_{1}}^{(1)}(w_1) \\
\vdots & \vdots & \ddots & \vdots \\
Q_{\bm{n}}^{(1)}(w_M) & Q_{\bm{n}-\bm{e}_{1}}^{(1)}(w_M) & \cdots & Q_{\bm{n}-M\bm{e}_{1}}^{(1)}(w_M) \\
\end{pmatrix}},}
where $Q_{\bm{n}}^{(1)}$ is given by \eqref{eq:q gt type II} and \eqref{eq:q gt type II measures} with $j = 1$ and $w_{1,k} = w_k$. Again, there is only possible sequence starting at $\bm{n}$.} 

\textnormal{On the other hand, for type I polynomials we get the more complicated formula
\nm{eq:partial gt type I}{& \widecheck{A}_{\bm{n}+M\bm{e}_1}^{(1)}(x) = \det{\begin{pmatrix}
A_{\bm{n}}^{(1)}(x) & A_{\bm{n}+\bm{s}_{1}}^{(1)}(x) & \cdots & A_{\bm{n}+\bm{s}_{rM}}^{(1)}(x) \\
A_{\bm{n}}^{(2)}(\bm{w}) & A_{\bm{n}+\bm{s}_{1}}^{(2)}(\bm{w}) & \cdots & A_{\bm{n}+\bm{s}_{rM}}^{(2)}(\bm{w}) \\
\vdots & \vdots & \ddots & \vdots \\
A_{\bm{n}}^{(r)}(\bm{w}) & A_{\bm{n}+\bm{s}_{1}}^{(r)}(\bm{w}) & \cdots & A_{\bm{n}+\bm{s}_{rM}}^{(r)}(\bm{w}) \\
B_{\bm{n}}(\bm{w}) & B_{\bm{n}+\bm{s}_{1}}(\bm{w}) & \cdots & B_{\bm{n}+\bm{s}_{rM}}(\bm{w})
\end{pmatrix}}, \\
& 
\widecheck{A}_{\bm{n}+M\bm{e}_1}^{(j)}(x) = \Phi(x)^{-1}\det{\begin{pmatrix}
A_{\bm{n}}^{(j)}(x) & A_{\bm{n}+\bm{s}_{1}}^{(j)}(x) & \cdots & A_{\bm{n}+\bm{s}_{rM}}^{(j)}(x) \\
A_{\bm{n}}^{(2)}(\bm{w}) & A_{\bm{n}+\bm{s}_{1}}^{(2)}(\bm{w}) & \cdots & A_{\bm{n}+\bm{s}_{rM}}^{(2)}(\bm{w}) \\
\vdots & \vdots & \ddots & \vdots \\
A_{\bm{n}}^{(r)}(\bm{w}) & A_{\bm{n}+\bm{s}_{1}}^{(r)}(\bm{w}) & \cdots & A_{\bm{n}+\bm{s}_{rM}}^{(r)}(\bm{w}) \\
B_{\bm{n}}(\bm{w}) & B_{\bm{n}+\bm{s}_{1}}(\bm{w}) & \cdots & B_{\bm{n}+\bm{s}_{rM}}(\bm{w})
\end{pmatrix}}, \qquad j = 2,\dots,r,}
where $\{\bm{n}+\bm{s}_j\}_{j=0}^{rM}$ is any admissible sequence from $\bm{n}$ towards $\bm{n} + M\bm{1} = (n_1+M,\dots,n_r+M)$.
Here we can use the same $B_{\bm{n}}$ as in \eqref{eq:q full gt type I} and \eqref{eq:q full gt type I measures}. }

\subsection{Various One-Step Transforms}\hfill\\

We consider measures/functionals perturbed by degree $1$ polynomials. More specifically, we say that a one-step Christoffel transform is a system of the form
\nm{eq:general one step ct}{(\widehat{\mu}_1,\dots,\widehat{\mu}_r) = (\Phi_1\mu_1,\dots\Phi_r\mu_r), \qquad \deg\Phi_j \leq 1, \qquad j = 1,\dots,r,}
and a one-step Geronimus transform is a system $(\widecheck{\mu}_1,\dots,\widecheck{\mu}_r)$ satisfying
\nm{eq:general one step gt}{(\Psi_1\widecheck{\mu}_1,\dots,\Psi_r\widecheck{\mu}_r) = (\mu_1,\dots,\mu_r), \qquad \deg\Psi_j \leq 1, \qquad j = 1,\dots,r.}
It is likely that the one-step transforms are going to be the most often used, so for the convenience of the reader we specialize our formulas for some of the simplest and most useful cases. 

\subsubsection{\textnormal{\textbf{The simplest one-step Christoffel transforms}}}\label{ex:simpleChr}  \hfill

\textnormal{First, we note that with $\Phi_1(x) = \cdots = \Phi_r(x) = x - z_1$ in \eqref{eq:general one step ct}, 
the monic type II polynomials are given by
\nm{eq:one step full ct type II}{\widehat{P}_{\bm{n}}(x) & = (x-z_1)^{-1}P_{\bm{n}}(z_1)^{-1}\det{\begin{pmatrix}
    P_{\bm{n}+\bm{e}_k}(x) & P_{\bm{n}}(x) \\
    P_{\bm{n}+\bm{e}_k}(z_1) & P_{\bm{n}}(z_1)
\end{pmatrix}} \\ & = (x-z_1)^{-1}\left(P_{\bm{n}+\bm{e}_k}(x)-\frac{P_{\bm{n}+\bm{e}_k}(z_1)}{P_{\bm{n}}(z_1)}P_{\bm{n}}(x)\right), \qquad k = 1,\dots,r.}
This is very similar to \eqref{eq:one step ct thm one measure}, and identical in the case $r = 1$.}

\textnormal{
For type I polynomials, for the case $\Phi_1(x) = x - z_1$ and $\Phi_2(x) = \cdots = \Phi_r(x) = 1$ 
we get the similar formula
\nm{eq:one step partial ct type I}{& \widehat{A}_{\bm{n}}^{(1)}(x) = (x-z_1)^{-1}A_{\bm{n}+\bm{e}_1}^{(1)}(z_1)^{-1}\det{\begin{pmatrix}
        A_{\bm{n}}^{(1)}(x) & A_{\bm{n}+\bm{e}_1}^{(1)}(x) \\
        A_{\bm{n}}^{(1)}(z_1) & A_{\bm{n}+\bm{e}_1}^{(1)}(z_1)
    \end{pmatrix}}, \\
    & \widehat{A}_{\bm{n}}^{(j)}(x) = A_{\bm{n}+\bm{e}_1}^{(1)}(z_1)^{-1}\det{\begin{pmatrix}
        A_{\bm{n}}^{(j)}(x) & A_{\bm{n}+\bm{e}_1}^{(j)}(x) \\
        A_{\bm{n}}^{(1)}(z_1) & A_{\bm{n}+\bm{e}_1}^{(1)}(z_1)
    \end{pmatrix}}, \qquad j = 2,\dots,r.}
In the case $r = 1$ this is also \eqref{eq:one step ct thm one measure}, but with a different normalization on the polynomials. \eqref{eq:one step partial ct type I} corresponds to $\bm{m} = \bm{n}$ in Theorem \ref{thm:type I thm}. More generally, the choice $\bm{m} = \bm{n} + \bm{e}_1 - \bm{e}_k$ works for any $k = 1,\dots,r$, and then we have the determinantal formula
\nm{eq:one step general partial ct type I}{& \widehat{A}_{\bm{n}}^{(1)}(x) = (x-z_1)^{-1}A_{\bm{n}+\bm{e}_1}^{(1)}(z_1)^{-1}\det{\begin{pmatrix}
        A_{\bm{n}+\bm{e}_1-\bm{e}_k}^{(1)}(x) & A_{\bm{n}+\bm{e}_1}^{(1)}(x) \\
        A_{\bm{n}+\bm{e}_1-\bm{e}_k}^{(1)}(z_1) & A_{\bm{n}+\bm{e}_1}^{(1)}(z_1)
    \end{pmatrix}}, \\
    & \widehat{A}_{\bm{n}}^{(j)}(x) = A_{\bm{n}+\bm{e}_1}^{(1)}(z_1)^{-1}\det{\begin{pmatrix}
        A_{\bm{n}+\bm{e}_1-\bm{e}_k}^{(j)}(x) & A_{\bm{n}+\bm{e}_1}^{(j)}(x) \\
        A_{\bm{n}+\bm{e}_1-\bm{e}_k}^{(1)}(z_1) & A_{\bm{n}+\bm{e}_1}^{(1)}(z_1)
    \end{pmatrix}}, \qquad j = 2,\dots,r.}}

\subsubsection{\textnormal{\textbf{The simplest one-step Geronimus transforms}}}\label{ss:one-stepGeronimus}  \hfill

\textnormal{
The formulas in this example are dual to the ones in the previous example. 
For the one-step Geronimus transform with $\Psi_1(x) = \cdots = \Psi_r(x) = x - w_1$
, the type I polynomials are given by 
\nm{eq:one step full gt type I}{
\widecheck{A}_{\bm{n}}^{(j)}(x)
=A_{\bm{n}-\bm{e}_k}^{(j)}(x)  -  \frac{B_{\bm{n}-\bm{e}_k}(w_1)}{B_{\bm{n}}(w_1)} A_{\bm{n}}^{(j)}(x), \qquad j = 1,\dots,r,}
assuming $n_k > 0$. Here, $B_{\bm{n}}(w_1)$ is given by
\nm{eq:q one step full gt type I}{B_{\bm{n}}(w_1) = \sum_{j = 1}^r \widecheck{\mu}_j[A_{\bm{n}}^{(j)}(x)], \qquad \bm{n} \in \N^r.}
Meanwhile, in the case $\Psi_1(x) = x - w_1$ and $\Psi_2(x) = \cdots = \Psi_r(x) = 1$
, we get the type II formula
\nm{eq:one step partial gt type II}{
\widecheck{P}_{\bm{n}}(x)
= 
P_{\bm{n}-\bm{e}_1+\bm{e}_k}(x) - \frac{Q^{(1)}_{\bm{n}-\bm{e}_1+\bm{e}_k}(w_1)}{Q^{(1)}_{\bm{n}-\bm{e}_1}(w_1)} P_{\bm{n}-\bm{e}_1}(x), \qquad k = 1,\dots,r,}
where $Q^{(1)}_{\bm{n}}(w_1)$ is given by
\nm{eq:q one step partial gt type II}{Q^{(1)}_{\bm{n}}(w_1) = \widecheck{\mu}_1[P_{\bm{n}}(x)], \qquad \bm{n} \in \N^r.}
}

\subsubsection{\textnormal{\textbf{One-step Christoffel transform, type I polynomials}}}\label{ex:generalChr}  \hfill

\textnormal{For the more general one-step Christoffel transform, 
with $\Phi_j(x) = x-z_j$ for $z_j \in \C$ and $j = 1,\dots,r$ in \eqref{eq:general one step ct}
(note that $z_j$'s here are allowed to coincide), we get the determinantal formula
\nm{eq:one step general ct type I}{\widehat{A}_{\bm{n}}^{(j)}(x) = \Phi_j(x)^{-1}\det{\begin{pmatrix}
        A_{\bm{n}}^{(j)}(x) & A_{\bm{n}+\bm{e}_1}^{(j)}(x) & \cdots & A_{\bm{n}+\bm{e}_r}^{(j)}(x) \\
        A_{\bm{n}}^{(1)}(z_1) & A_{\bm{n}+\bm{e}_1}^{(1)}(z_1) & \cdots & A_{\bm{n}+\bm{e}_r}^{(1)}(z_1) \\
        \vdots & \vdots & \ddots & \vdots \\
        A_{\bm{n}}^{(r)}(z_r) & A_{\bm{n}+\bm{e}_1}^{(r)}(z_r) & \cdots & A_{\bm{n}+\bm{e}_r}^{(r)}(z_r) \\
    \end{pmatrix}}, \qquad j = 1,\dots,r.}
Here we chose the 
frame from $\bm{n}$ towards $\bm{n}+\bm{1} = (n_1+1,\dots,n_r+1)$, which corresponds to the sequence of the nearest neighbours in the one-step case. This is probably the nicest choice of sequence, but we remark that the sequence can also be chosen along the step-line if the index is on the step-line (more generally, we note that we can restrict all our determinantal formulas to the step-line when $\deg{\Phi_1} = \cdots = \deg{\Phi_r}$). }

\subsubsection{\textnormal{\textbf{One-step Geronimus transform, type II polynomials}}}  \hfill

\textnormal{In a similar fashion to the previous example, 
the type II formula for the one-step Geronimus transform 
with $\Psi_j(x) = x-w_j$ in \eqref{eq:general one step gt}
is given by
\nm{eq:one step general gt type II}{\widecheck{P}_{\bm{n}}(x) = \det{\begin{pmatrix}
P_{\bm{n}}(x) & P_{\bm{n}-\bm{e}_{1}}(x) & \cdots & P_{\bm{n}-\bm{e}_{r}}(x) \\
Q_{\bm{n}}^{(1)}(w_1) & Q_{\bm{n}-\bm{e}_{1}}^{(1)}(w_1) & \cdots & Q_{\bm{n}-\bm{e}_{r}}^{(1)}(w_1) \\
\vdots & \vdots & \ddots & \vdots \\
Q_{\bm{n}}^{(r)}(w_r) & Q_{\bm{n}-\bm{e}_{1}}^{(r)}(w_r) & \cdots & Q_{\bm{n}-\bm{e}_{r}}^{(r)}(w_r) \\
\end{pmatrix}}.}
Here we have 
\nm{eq:q one step gt type II functional}{Q_{\bm{n}}^{(j)}(w_j) = \widecheck{\mu}_j[P_{\bm{n}}(x)], \qquad \bm{n} \in \N^r.} 
If $n_j = 0$ then $\bm{n} - \bm{e}_j \notin \N^r$, and according to Theorem \ref{thm:type II thm} we can put $Q_{\bm{n}-\bm{e}_j}^{(j)}(w_j) = 1$. Equivalently, we can simply remove the row and column containing 
$Q_{\bm{n}-\bm{e}_j}^{(j)}(w_j)$ in \eqref{eq:one step general gt type II}.
If $\mu_j$ is a measure and $w_j \notin \supp{(\mu_j)}$ we have Geronimus transforms
\nm{eq:one step gt moprl measure}{\int f(x)d\widecheck{\mu}_j(x) = \int \frac{f(x)}{x-w_j}d\mu_j(x) + c_jf(w_j),}
and then we get
\nm{eq:q one step gt type II}{Q_{\bm{n}}^{(j)}(w_j) = \int \frac{P_{\bm{n}}(x)}{x-w_j}d\mu_j(x) + c_jP_{\bm{n}}(w_j), \qquad \bm{n} \in \N^r.}}

\subsubsection{\textnormal{\textbf{One-step Christoffel transform, type II polynomials}}}  \hfill

\textnormal{
The formula here is more complicated than the one in the type I case. We choose to include it in the case of a system of two measures $(\mu_1,\mu_2)$, but it will be clear how it generalizes to any system of measures/functionals. Here we have $(\widehat{\mu}_1,\widehat{\mu}_2) = (\Phi_1\mu_1,\Phi_2\mu_2)$, with $\Phi_1(x) = x-z_1$ and $\Phi_2(x) = x-z_2$. To apply Theorem \ref{thm:type II thm} we note that we have
\nm{eq:general one step ct two measures}{
(x-z_2)d\widehat{\mu}_1(x) = (x-z_1)(x-z_2)d\mu_1(x), \qquad (x-z_1)d\widehat{\mu}_2(x) = (x-z_1)(x-z_2)d\mu_2(x).
}
The Geronimus transforms $\widecheck{\mu}_1$ and $\widecheck{\mu}_2$ satisfy
\nm{eq:gt general one step ct}{(x-z_1)(x-z_2)d\widecheck{\mu}_1(x) = (x-z_1)d\mu_1(x), \qquad (x-z_1)(x-z_2)d\widecheck{\mu}_1(x) = (x-z_2)d\mu_1(x).}
We assume $z_2 \notin \supp(\mu_1)$ and $z_1 \notin \supp(\mu_2)$, and define Geronimus transforms $\widecheck{\mu}_1$ and $\widecheck{\mu}_2$ by
\nm{eq:gt choice}{d\widecheck{\mu}_1(x) = (x-z_2)^{-1}d\mu_1(x), \qquad d\widecheck{\mu}_2(x) = (x-z_1)^{-1}d\mu_2(x).}
We then have
\nm{eq:q general one step ct type II}{Q_{n,m}^{(1)}(z_2) = \int\frac{P_{n,m}(x)}{x-z_2}d\mu_1(x), \qquad Q_{n,m}^{(2)}(z_1) = \int\frac{P_{n,m}(x)}{x-z_1}d\mu_2(x), \qquad (n,m) \in \N^2,}
and $(x-z_1)(x-z_2)\widehat{P}_{n,m}(x)$ is given by the determinant
\nm{eq:general one step ct type II}{\det{\begin{pmatrix}
    P_{n+1,m+1}(x) & P_{n+1,m}(x) & P_{n,m+1}(x) & P_{n+1,m-1}(x) & P_{n-1,m+1}(x) \\
    P_{n+1,m+1}(z_1) & P_{n+1,m}(z_1) & P_{n,m+1}(z_1) & P_{n+1,m-1}(z_1) & P_{n-1,m+1}(z_1) \\
    P_{n+1,m+1}(z_2) & P_{n+1,m}(z_2) & P_{n,m+1}(z_2) & P_{n+1,m-1}(z_2) & P_{n-1,m+1}(z_2) \\
    Q_{n+1,m+1}^{(2)}(z_1) & Q_{n+1,m}^{(2)}(z_1) & Q_{n,m+1}^{(2)}(z_1) & Q_{n+1,m-1}^{(2)}(z_1) & Q_{n-1,m+1}^{(2)}(z_1) \\
    Q_{n+1,m+1}^{(1)}(z_2) & Q_{n+1,m}^{(1)}(z_2) & Q_{n,m+1}^{(1)}(z_2) & Q_{n+1,m-1}^{(1)}(z_2) & Q_{n-1,m+1}^{(1)}(z_2)
\end{pmatrix}}.}
If $n = 0$ then we remove the row and column containing $Q_{n-1,m+1}^{(1)}(z_2)$, and if $m = 0$ we remove the row and column containing $Q_{n+1,m-1}^{(2)}(z_1)$. The sequence starts at $(n+1,m+1)$, but we could also choose a sequence starting at $(n+2,m)$ or $(n,m+2)$. Here we chose the decreasing frame from $(n+1,m+1)$ towards $(n-1,m-1)$ as our sequence of multi-indices. It is also possible to choose a sequence 
on the step-line, when $(n,m)$ is on the step-line. }

\subsection{Uvarov Transforms}\hfill\\

Uvarov perturbations are also examples of rational perturbations of functionals. We first give type II determinantal formulas for a simple special case, and then give a type II formula for the most general case. 

\subsubsection{\textnormal{\textbf{One-step Uvarov transform, type II polynomials}}}  \hfill

\textnormal{We assume $\mu_1$ and $\mu_2$ are measures with $z_1 \notin \supp(\mu_1)\cup\supp(\mu_2)$ and consider Uvarov perturbations $(\widetilde{\mu}_1,\widetilde{\mu}_2)$ of the type 
\nm{eq:uvarov two measures}{\widetilde{\mu}_1 = \mu_1 + c_1\delta_{z_1}, \qquad \widetilde{\mu}_2 = \mu_2 + c_2\delta_{z_1},} 
where at most one of $c_1$ and $c_2$ are $0$. Theorem \ref{thm:type I thm} and Theorem \ref{thm:type II thm} can be applied to these systems as well, since we have
\nm{eq:uvarov property two measures}{\int f(x)(x-z_1)d\widetilde{\mu}_j(x) = \int f(x)(x-z_1)d\mu_j(x), \qquad j = 1,2. 
} 
The Geronimus transforms $\widecheck{\mu}_1$ and $\widecheck{\mu}_2$ have to satisfy
\nm{eq:gt uvarov}{\widecheck{\mu}_j[(x-z_1)P(x)] = \int P(x)d\mu_j(x) + c_jP(z_1), \qquad j = 1,2. 
}
We make the choice
\nm{eq:gt choice uvarov}{\widecheck{\mu}_j[ P(x)] = \int \frac{P(x)}{x-z_1}d\mu_j(x) + c_jP'(z_1), \qquad j = 1,2.}
Here, $\widecheck{\mu}_j = (x-z_1)^{-1}\mu_j - c_j\delta_{z_1}'$, where $\delta_{z_1}'$ is the functional $P \mapsto -P'(z_1)$, and $(x-z_1)^{-1}\mu_j$ is the functional generated by the measure $(x-z_1)^{-1}d\mu(x)$. Hence $\widecheck{\mu}_1$ and $\widecheck{\mu}_2$ are functionals but not measures. We then have 
\nm{eq:q one step uvarov}{Q_{n,m}^{(j)}(z_1) = \int \frac{P_{n,m}(x)}{x-z_1}d\mu_j(x) + c_j P_{n,m}'(z_1), \qquad j = 1,2, 
\qquad (n,m) \in \N^2,}
and we get a type II formula of the form
\nm{eq:one step uvarov type II}{\widetilde{P}_{n,m}(x) = (x-z_1)^{-1}\det{\begin{pmatrix}
    P_{n+1,m}(x) & P_{n,m}(x) & P_{n-1,m}(x) & P_{n,m-1}(x) \\
    P_{n+1,m}(z_1) & P_{n,m}(z_1) & P_{n-1,m}(z_1) & P_{n,m-1}(z_1) \\
    Q_{n+1,m}^{(1)}(z_1) & Q_{n,m}^{(1)}(z_1) & Q_{n-1,m}^{(1)}(z_1) & Q_{n,m-1}^{(1)}(z_1) \\
    Q_{n+1,m}^{(2)}(z_1) & Q_{n,m}^{(2)}(z_1) & Q_{n-1,m}^{(2)}(z_1) & Q_{n,m-1}^{(2)}(z_1)
\end{pmatrix}}.}
If $n = 0$ we remove the row and column containing $Q_{n-1,m}^{(1)}(z_1)$, and if $m = 0$ we remove the row and column containing $Q_{n,m-1}^{(2)}(z_1)$. Different choices of sequences can also restrict all polynomials to any choice of path. Note that the sequence chosen here is neither a path nor a frame, but rather a combination of the two, similar to sequence in the middle of the figure in Theorem \ref{thm:type II thm}. 
}

\subsubsection{\textnormal{\textbf{General Uvarov transform, type II polynomials}}}\label{ex:fullUvarov}  \hfill

\textnormal{We say that a general Uvarov transform is a system 
\nm{eq:general uvarov}{(\widetilde{\mu}_1,\dots,\widetilde{\mu}_r) = \bigg(\mu_1 + \sum_{k = 1}^N c_{1,k}\delta_{z_k},\dots,\mu_r + \sum_{k = 1}^N c_{r,k}\delta_{z_k}\bigg), \qquad z_1,\dots,z_N \in \C\setminus\cup_{j=1}^r \supp\mu_j.}
Note that the location of the point masses are the same for each measure, but the weights are different. If we put different weights to $0$ for different $\mu_j$'s we then have the case where we add different points to each measure. For example, with $N = r$, $c_{j,j} = c_j$, and $c_{j,k} = 0$ when $j \neq k$, we get the Uvarov transforms on the form
\nm{eq:general one step uvarov}{(\widetilde{\mu}_1,\dots,\widetilde{\mu}_r) = (\mu_1 + c_1\delta_{z_1},\dots,\mu_r + c_r\delta_{z_r})}} 

\textnormal{The systems \eqref{eq:general uvarov} satisfy 
\nm{eq:uvarov property}{(\Phi\widetilde{\mu}_1,\dots,\Phi\widetilde{\mu}_r) = (\Phi\mu_1,\dots,\Phi\mu_r),}
for the polynomial $\Phi(x) = \prod_{k = 1}^{N}(x-z_k)$. From here we can apply both Theorem \ref{thm:type I thm} and Theorem \ref{thm:type II thm} to these systems. If we allow $\Phi$ to have zeros of higher multiplicity then we can also add derivatives of $\delta_{z_k}$'s. }

\textnormal{The type II formula for systems satisfying \eqref{eq:uvarov property} is
\nm{eq:det formula uvarov type II}{\widetilde{P}_{\bm{n}} = \Phi(x)^{-1}\det{\begin{pmatrix}
        P_{\bm{n}+N\bm{e}_1}(x) & P_{\bm{n}+N\bm{e}_1-\bm{s}_1}(x) & \dots & P_{\bm{n}+N\bm{e}_1-\bm{s}_{(r+1)N}}(x) \\
        P_{\bm{n}+N\bm{e}_1}(\bm{z}) & P_{\bm{n}+N\bm{e}_1-\bm{s}_1}(\bm{z}) & \dots & P_{\bm{n}+N\bm{e}_1-\bm{s}_{(r+1)N}}(\bm{z}) \\
        Q_{\bm{n}+N\bm{e}_1}^{(1)}(\bm{z}) & Q_{\bm{n}+N\bm{e}_1-\bm{s}_1}^{(1)}(\bm{z}) & \dots & Q_{\bm{n}+N\bm{e}_1-\bm{s}_{(r+1)N}}^{(1)}(\bm{z}) \\
        \vdots & \vdots & \ddots & \vdots \\
        Q_{\bm{n}+N\bm{e}_1}^{(r)}(\bm{z}) & Q_{\bm{n}+N\bm{e}_1-\bm{s}_1}^{(r)}(\bm{z}) & \dots & Q_{\bm{n}+N\bm{e}_1-\bm{s}_{(r+1)N}}^{(r)}(\bm{z}) \\
    \end{pmatrix}}.}
    {It is easy to check that $\Phi\widecheck{\mu}_j = \widetilde{\mu}_j$ is satisfied by
    \nm{eq:}{\widecheck\mu_j = \frac{1}{\Phi}\mu_j - \sum_{l=1}^N \frac{c_{j,l}}{\Phi'(z_l)} \delta'_{z_l} + \sum_{l = 1}^N d_{j,l}\delta_{z_l}, \qquad j = 1,\dots,r,}
    for any choice of weights $d_{j,l}$, where $(1/\Phi)\mu_j$ is the functional given by the measure $\Phi(x)^{-1}d\mu_j(x)$, and $\delta_{z}'[P(x)] = -P'(z)$. Then the $Q$'s are given by
    \m{
        Q_{\bm{n}}^{(j)}(z_k) & = \widecheck{\mu}_j\left[P_{\bm{n}}(x)\frac{\Phi(x)}{x-z_k}\right] \\
    & = 
    \int \frac{P_{\bm{n}}(x)}{x-z_k}d\mu_j(x) - \sum_{l = 1}^N\frac{c_{j,l}}{\Phi'(z_l)}\delta_{z_l}'\left[P_{\bm{n}}(x)\frac{\Phi(x)}{x-z_k}\right] + d_{j,k}P_{\bm{n}}(z_k)\Phi'(z_k), 
    }
    for $ j= 1,\ldots,r$ and $k = 1,\dots,N$.
    A quick calculation shows that
    \m{-\delta_{z_l}'\left[P_{\bm{n}}(x)\frac{\Phi(x)}{x - z_k}\right] & = \delta_{z_l}\left[P_{\bm{n}}'(x)\frac{\Phi(x)}{x-z_k} + P_{\bm{n}}(x)\frac{\Phi'(x)(x-z_k) - \Phi(x)}{(x-z_k)^2}\right] \\ & = \begin{cases}
        \Phi'(z_l)P_{\bm{n}}(z_l)(z_l-z_k)^{-1}, & k \neq l, \\
        \Phi'(z_k)P_{\bm{n}}'(z_k) + \frac{1}{2}P_{\bm{n}}(z_k)\Phi''(z_k), & k = l.
    \end{cases}}
    With the choice $d_{j,k} = -\frac{1}{2}c_{j,k}\Phi''(z_l)/\Phi'(z_l)^2$ we then end up with
    \nm{eq:q general uvarov type II}{Q_{\bm{n}}^{(j)}(z_k) & = \int \frac{P_{\bm{n}}(x)}{x-z_k}d\mu_j(x) + c_{j,k}P_{\bm{n}}'(z_k) + \sum_{l \neq k} c_{j,l}\frac{P_{\bm{n}}(z_k)}{z_l-z_k}, \qquad j = 1,\dots,r.}
    } 
    }

    \textnormal{{Now observe that in the determinant \eqref{eq:det formula uvarov type II} we can subtract multiples of the rows with $P_{\bm{n}}$'s from the rows with $Q_{\bm{n}}^{(j)}$'s. This shows that \eqref{eq:det formula uvarov type II} holds 
    with $Q_{\bm{n}}^{(j)}(z_k)$ replaced by the simpler expression
    \nm{eq:simpler q general uvarov type II}{
    \int \frac{P_{\bm{n}}(x)}{x-z_k}d\mu_j(x) + c_{j,k} P_{\bm{n}}'(z_k), \qquad j = 1,\dots,r, \qquad k = 1,\dots,N.}
    In particular, when $r = 1$ we get \eqref{eq:pureUvarov}.}}
    


\textnormal{
The sequence of multi-indices $\set{\bm{n}+N\bm{e}_1-\bm{s}_j}_{j = 0}^{(r+1)N}$ is admissible from $\bm{n}+N\bm{e}_1$ to $\bm{n}-N\bm{1}$, for example a backwards frame. Note that if we remove the set of rows and columns containing $Q_{\bm{n}+N\bm{e}_1-\bm{s}_i}^{(1)}(\bm{z})$ for $\bm{n}+N\bm{e}_1-\bm{s}_i \in \set{\bm{n}-j\bm{e}_1}_{j = 1}^{N}$, we get \eqref{eq:partial ct type II} (with a shift, and possibly different $Q$'s). If we instead remove the rows and columns containing $P_{\bm{n}+N\bm{e}_1-\bm{s}_i}(\bm{z})$ for $\bm{n}+N\bm{e}_1-\bm{s}_i$ in $\set{\bm{n}+j\bm{e}_1}_{j = 1}^{N}$, we get something more similar to \eqref{eq:general gt type II}. If we remove every row and column containing $Q$'s for indices in $\bigcup_{i = 1}^{r}\set{\bm{n}-j\bm{e}_i}_{j = 1}^{N}$, we get \eqref{eq:full ct type II} with $\bm{z}_1 = \cdots = \bm{z}_r = \bm{z}$. }


\section{Example: Rational perturbations of multiple Laguerre  polynomials of the second kind}\label{ss:Laguerre}

Theorems~\ref{thm:type I thm} and~\ref{thm:type II thm} apply for any system $\bm{\mu}$ with sufficiently many normal indices. If one has an explicit form of type I or type II multiple orthogonal polynomials $\bm{\mu}$ then orthogonal polynomials of the perturbed system $\widetilde{\mu}$ is just a matter of calculation. Let us illustrate on an example of perturbations of the multiple Laguerre  polynomials of the second kind for the case of two measures. Let $\bm{\mu} = (\mu_1,\mu_2)$ with
\begin{equation}
    d\mu_1(x) = x^{\alpha} e^{-\beta_1 x} \,dx, \qquad d\mu_2(x) = x^{\alpha}e^{-\beta_2 x} \,dx,
\end{equation}
on $[0,\infty)$ for some $\alpha>-1$, $\beta_1,\beta_2>0$, $\beta_1\ne \beta_2$.

Recall~\cite{ABVA,VAC01,Ismail} that the monic type II polynomials for such systems are given by
\begin{align}
    \label{eq:LagII}
    P_{\bm{n}}(x) =& \sum_{j=0}^{n_1} \sum_{k=0}^{n_2} c_{\bm{n};j,k} x^{j+k}, 
    \\
    \label{eq:LagtypeIIconst}
    \mbox{where }
    &c_{\bm{n};j,k}= \frac{(-1)^{|\bm{n}|+j+k} }{\beta_1^{n_1-j} \beta_2^{n_2-k} } 
    \frac{n_1! \,n_2!}{j!\,k!}
    \binom{\alpha+|\bm{n}|}{n_1-j} \binom{\alpha+n_2+j}{n_2-k} ,
\end{align}
and the type I polynomials were recently computed in~\cite{BDFM23,BDFM25} to be
\begin{align}
\label{eq:LagI}
    A_{\bm{n}}^{(1)}(x) =& \sum_{j=0}^{n_1-1} c^{(1)}_{\bm{n};j} x^{j}, 
    \\
    \label{eq:LagtypeIconst}
    \mbox{where }
    &c^{(1)}_{\bm{n};j}= \frac{(-1)^{|\bm{n}|-1}  \beta_1^{\alpha+|\bm{n}|} \beta_2^{n_2} (\alpha+n_2+1)_{n_1-1}}{(n_1-1)! (\beta_1-\beta_2)^{n_2} \, \Gamma(\alpha+|\bm{n}|)} 
    \sum_{k=0}^{n_1-1-j} \frac{(-n_1+1)_{j+k} (n_2)_k}{(\alpha+n_2+1)_{j+k} j!\, k!} 
     \frac{\beta_1^{j+k}}{(\beta_1-\beta_2)^k},
     \\
     \label{eq:LagI2}
     A_{\bm{n}}^{(2)}(x) =& \sum_{j=0}^{n_2-1} c^{(2)}_{\bm{n};j} x^{j}, 
    \\
    \label{eq:LagtypeIconst2}
    \mbox{where }
    &c^{(2)}_{\bm{n};j}= \frac{(-1)^{|\bm{n}|-1}  \beta_2^{\alpha+|\bm{n}|} \beta_1^{n_1} (\alpha+n_1+1)_{n_2-1}}{(n_2-1)! (\beta_2-\beta_1)^{n_1} \, \Gamma(\alpha+|\bm{n}|)} 
    \sum_{k=0}^{n_2-1-j} \frac{(-n_2+1)_{j+k} (n_1)_k}{(\alpha+n_1+1)_{j+k} j!\, k!} 
     \frac{\beta_2^{j+k}}{(\beta_2-\beta_1)^k},
\end{align}
where $(x)_a = x (x+1)\ldots (x+a-1)$ is the Pochhammer symbol. 

Now let us consider $\widetilde{\bm\mu} = (\widetilde{\mu}_1,\widetilde{\mu}_2)$ given by
\begin{align}
    \widetilde{\mu}_1 &= x \mu_1, 
    \\
    \widetilde{\mu}_2 &= \frac{1}{x} \mu_2 + w \delta_0,
\end{align}
where $ \frac{1}{x} \mu_2$ stands for the pure Geronimus transform without a point mass (we require $\alpha>0$ so that $\widetilde{\mu}_2$ is a measure), $\delta_0$ is the point mass at $x=0$, and $w\in\bbC$ is an arbitrary weight. Explicitly, see~\eqref{eq:LagPert1}--\eqref{eq:LagPert2}.
Observe that even in the case $w=0$, the new system is no longer of the multiple Laguerre type. 

\begin{prop}\label{prop:ex1}
    Consider the system $\widetilde{\bm\mu} = (\widetilde{\mu}_1,\widetilde{\mu}_2)$ given by
\begin{align}\label{eq:LagPert1}
    d\widetilde\mu_1(x) & = x^{\alpha+1} e^{-\beta_1 x} \,dx,
    \\
    \label{eq:LagPert2}
    d\widetilde\mu_2(x) &= x^{\alpha-1} e^{-\beta_2 x} \,dx + w\delta_0,
\end{align}
on $[0,\infty)$ for some $\alpha>1$, $\beta_1,\beta_2>0$, $\beta_1\ne \beta_2$, $w\in\bbC$.

Then $\bm{n}=(n_1,n_2)$ is a normal index for  $\widetilde{\bm\mu}$ if and only if $D^{(II)}_{\bm{n}} \ne 0$, where
\begin{equation}\label{eq:LIID}
    D^{(II)}_{\bm{n}}=
    \det \begin{pmatrix}
         P_{\bm{n}}(0) & P_{\bm{n}-\bm{e}_2}(0) & P_{\bm{n}-2\bm{e}_2}(0) \\
         Q_{\bm{n}}(0) & Q_{\bm{n}-\bm{e}_2}(0) & Q_{\bm{n}-2\bm{e}_2}(0) \\
         R_{\bm{n}}(0) & R_{\bm{n}-\bm{e}_2}(0) & R_{\bm{n}-2\bm{e}_2}(0)
    \end{pmatrix},
\end{equation}
where $P_{\bm{n}}(x), Q_{\bm{n}}(x), R_{\bm{n}}(x)$ are given by~\eqref{eq:LagII},\eqref{eq:exQ}, \eqref{eq:exR}, respectively.
In this case the unique monic type II polynomial of $\widetilde{\bm\mu}$ is given by
\begin{equation}\label{eq:LIIdet}
    \widetilde{P}_{\bm{n}}(x) 
    =
    \frac{1}{D^{(II)}_{\bm{n}}} \frac{1}{x}
    \det
    \begin{pmatrix}
        P_{\bm{n}+\bm{e}_2}(x) & P_{\bm{n}}(x) & P_{\bm{n}-\bm{e}_2}(x) & P_{\bm{n}-2\bm{e}_2}(x) \\
        P_{\bm{n}+\bm{e}_2}(0) & P_{\bm{n}}(0) & P_{\bm{n}-\bm{e}_2}(0) & P_{\bm{n}-2\bm{e}_2}(0) \\
        Q_{\bm{n}+\bm{e}_2}(0) & Q_{\bm{n}}(0) & Q_{\bm{n}-\bm{e}_2}(0) & Q_{\bm{n}-2\bm{e}_2}(0) \\
        R_{\bm{n}+\bm{e}_2}(0) & R_{\bm{n}}(0) & R_{\bm{n}-\bm{e}_2}(0) & R_{\bm{n}-2\bm{e}_2}(0) 
    \end{pmatrix}.
\end{equation}
Here 
$P_{\bm{n}}(0) = c_{\bm{n};0,0}$, see~\eqref{eq:LagtypeIIconst}, and
\begin{align}
\label{eq:exQ}
    Q_{\bm{n}}(0) & = \int_0^\infty \frac{P_{\bm{n}}(x)}{x} d\mu_2 = 
    \sum_{j=0}^{n_1} \sum_{k=0}^{n_2} c_{\bm{n};j,k} \frac{\Gamma(\alpha+j+k)}{\beta_2^{\alpha+j+k}},
    \\
    \label{eq:exR}
    R_{\bm{n}}(0) & = \int_0^\infty \frac{P_{\bm{n}}(x)}{x^2} d\mu_2 + c P_{\bm{n}}'(0)= 
    \sum_{j=0}^{n_1} \sum_{k=0}^{n_2} c_{\bm{n};j,k} \frac{\Gamma(\alpha+j+k-1)}{\beta_2^{\alpha+j+k-1}}
    + (c_{\bm{n};0,1} + c_{\bm{n};1,0}) w
    .
\end{align}
\end{prop}
\begin{proof}
In order to get the formulas, we apply Theorem~\ref{thm:type II thm} with $\Phi(x) = x$, $N=1$, $\Psi_1(x) = 1$, $\Psi_2(x)= x^2$, $\bm{M} = (0,2)$. We choose
\begin{align}
    \widecheck\mu_1 & = \mu_1, \\
    \label{eq:mu2}
    \widecheck\mu_2 & = \frac{1}{x^2} \mu_2 - w \delta_0 '.
\end{align}
Taking into account the multiplicity considerations from Section~\ref{general case} we obtain~\eqref{eq:LIIdet} and~\eqref{eq:LIID} 
with the last two rows consisting of quantities
\begin{align}
\label{eq:justTemp}
    Q_{\bm{n}}(0)&  = \widecheck\mu_2\Big[P_{\bm{n}}(x) \frac{x^2}{x}\Big] = \widecheck\mu_2[x P_{\bm{n}}(x)] = \int_0^\infty \frac{P_{\bm{n}}(x)}{x} d\mu_2(x)  + w P_{\bm{n}}(0), 
    \\
    \label{eq:exR2}
    R_{\bm{n}}(0) & = \widecheck\mu_2\Big[P_{\bm{n}}(x) \frac{x^2}{x^2}\Big] =\int_0^\infty \frac{P_{\bm{n}}(x)}{x^2} d\mu_2(x)   + w P'_{\bm{n}}(0).
\end{align}
Note that we can remove the $w P_{\bm{n}}(0)$ term in~\eqref{eq:justTemp} since we are able to subtract a multiple of the second row in~\eqref{eq:LIIdet} from the third row (this is a consequence of the overlapping root of $\Phi$ and $\Psi_2$ at $x=0$, and otherwise would not be legal here).


Finally, using~\eqref{eq:LagII} and
\begin{equation}
    \int_0^\infty x^s d\mu_2(x) = \int_0^\infty x^{\alpha+s} e^{-\beta_2 x} dx = \frac{\Gamma(\alpha+s+1)}{\beta_2^{\alpha+s+1}}, 
\end{equation}
for $s\ge -2$ (recall that $\alpha>1$), we obtain~\eqref{eq:exQ} and~\eqref{eq:exR}. 
\end{proof}
\begin{rem}
    Observe that we required $\alpha>1$ in the above statement. This was needed in order for our choice of~\eqref{eq:mu2} to be a measure and for the integral in~\eqref{eq:exR},~\eqref{eq:exR2} to make sense (note the presence of $\Gamma(\alpha-1)$ in the sum on the right-hand side of ~\eqref{eq:exR}.
    
    If $0<\alpha\le 1$, then the system of measures $\widetilde{\bm\mu} = (\widetilde{\mu}_1,\widetilde{\mu}_2)$ in~\eqref{eq:LagPert1}--\eqref{eq:LagPert2} is perfectly well defined, and Theorem~\ref{thm:type II thm} is applicable with the following modification. $\widecheck{\mu}_2$ has to be chosen differently, specifically as the functional $\widecheck{\mu}_2[P] = \int_0^\infty \frac{P(x)-P(0)}{x^2} d\mu_2(x) + w P'(0)$, see~\eqref{eq:one-step geronimus functional}. Then $x\widecheck{\mu}_2 = \widetilde{\mu}$ holds, and the rest of the argument goes through, with the only difference being that $R_{\bm{n}}(0)$ becomes
    \begin{equation}
           R_{\bm{n}}(0)  = \int_0^\infty \frac{P_{\bm{n}}(x)-P_{\bm{n}}(0)}{x^2} d\mu_2(x) + c P_{\bm{n}}'(0)= 
    \mathop{\underset{j+k\ne 0}{{\sum_{j=0}^{n_1} \sum_{k=0}^{n_2}}}}
      c_{\bm{n};j,k} \frac{\Gamma(\alpha+j+k-1)}{\beta_2^{\alpha+j+k-1}}
    + (c_{\bm{n};0,1} + c_{\bm{n};1,0}) w.
    \end{equation} 
\end{rem}

The exact same argument works for type I polynomials with the help of Theorem~\ref{thm:type I thm}. The case $0<\alpha\le 1$ does not cause any trouble here so we can immediately assume $\alpha>0$.

\begin{prop}\label{prop:ex2}
    Consider the system $\widetilde{\bm\mu} = (\widetilde{\mu}_1,\widetilde{\mu}_2)$ given by~\eqref{eq:LagPert1}, ~\eqref{eq:LagPert2}, for some $\alpha>0$, $\beta_1,\beta_2>0$, $\beta_1\ne \beta_2$, $w\in\bbC$.

Then $\bm{n}=(n_1,n_2)$ is a normal index for  $\widetilde{\bm\mu}$ if and only if $D^{(I)}_{\bm{n}} \ne 0$, where
\begin{equation}\label{eq:LID}
    D^{(I)}_{\bm{n}}=
    \det \begin{pmatrix}
        A^{(1)}_{\bm{n}}(0) & A^{(1)}_{\bm{n}+\bm{e}_1}(0) & A^{(1)}_{\bm{n}+2\bm{e}_1}(0) \\
        A^{(1)\,\prime}_{\bm{n}}(0) & A^{(1)\,\prime}_{\bm{n}+\bm{e}_1}(0) & A^{(1)\,\prime}_{\bm{n}+2\bm{e}_1}(0) \\
        B_{\bm{n}}(0) & B_{\bm{n}+\bm{e}_1}(0) & B_{\bm{n}+2\bm{e}_1}(0)
    \end{pmatrix},
\end{equation}
where $A^{(1)}_{\bm{n}}(x)$, $A^{(2)}_{\bm{n}}(x)$, and $B_{\bm{n}}(0)$ are given by~\eqref{eq:LagI},\eqref{eq:LagI2}, and \eqref{eq:exB}, respectively.
In this case the unique normalized type I polynomials of $\widetilde{\bm\mu}$ are given by
\begin{align}\label{eq:LIdet1}
    \widetilde{A}^{(1)}_{\bm{n}}(x) 
    &=
    \frac{1}{D^{(I)}_{\bm{n}}} \frac{1}{x^2}
    \det
    \begin{pmatrix}
        A^{(1)}_{\bm{n}-\bm{e}_1}(x) & A^{(1)}_{\bm{n}}(x) & A^{(1)}_{\bm{n}+\bm{e}_1}(x) & A^{(1)}_{\bm{n}+2\bm{e}_1}(x) \\
        A^{(1)}_{\bm{n}-\bm{e}_1}(0) & A^{(1)}_{\bm{n}}(0) & A^{(1)}_{\bm{n}+\bm{e}_1}(0) & A^{(1)}_{\bm{n}+2\bm{e}_1}(0) \\
        A_{\bm{n}-\bm{e}_1}^{(1)\,\prime}(0) & A^{(1)\,\prime}_{\bm{n}}(0) & A^{(1)\,\prime}_{\bm{n}+\bm{e}_1}(0) & A^{(1)\,\prime}_{\bm{n}+2\bm{e}_1}(0) \\
        B_{\bm{n}-\bm{e}_1}(0) & B_{\bm{n}}(0) & B_{\bm{n}+\bm{e}_1}(0) & B_{\bm{n}+2\bm{e}_1}(0) 
    \end{pmatrix},
    \\
    \label{eq:LIdet2}
    \widetilde{A}^{(2)}_{\bm{n}}(x) 
    &=
    \frac{1}{D^{(I)}_{\bm{n}}} 
    \det
     \begin{pmatrix}
        A^{(2)}_{\bm{n}-\bm{e}_1}(x) & A^{(2)}_{\bm{n}}(x) & A^{(2)}_{\bm{n}+\bm{e}_1}(x) & A^{(2)}_{\bm{n}+2\bm{e}_1}(x) \\
        A^{(1)}_{\bm{n}-\bm{e}_1}(0) & A^{(1)}_{\bm{n}}(0) & A^{(1)}_{\bm{n}+\bm{e}_1}(0) & A^{(1)}_{\bm{n}+2\bm{e}_1}(0) \\
        A_{\bm{n}-\bm{e}_1}^{(1)\,\prime}(0) & A^{(1)\,\prime}_{\bm{n}}(0) & A^{(1)\,\prime}_{\bm{n}+\bm{e}_1}(0) & A^{(1)\,\prime}_{\bm{n}+2\bm{e}_1}(0) \\
        B_{\bm{n}-\bm{e}_1}(0) & B_{\bm{n}}(0) & B_{\bm{n}+\bm{e}_1}(0) & B_{\bm{n}+2\bm{e}_1}(0) 
    \end{pmatrix}.
\end{align}
Here 
$A^{(1)}_{\bm{n}}(0) = c_{\bm{n};0}^{(1)}$, $A^{(1)\,\prime}_{\bm{n}}(0) = c_{\bm{n};1}^{(1)}$, see~\eqref{eq:LagtypeIconst}, and
\begin{align}
\label{eq:exB}
    B_{\bm{n}}(0) & = \int_0^\infty \frac{A^{(1)}_{\bm{n}}(x)}{x} d\mu_1 + \int_0^\infty \frac{A^{(2)}_{\bm{n}}(x)}{x} d\mu_2 +w A^{(2)}_{\bm{n}}(0) 
    \\
    & = 
    \sum_{j=0}^{n_1-1} c^{(1)}_{\bm{n};j} \frac{\Gamma(\alpha+j)}{\beta_1^{\alpha+j}} + \sum_{j=0}^{n_2-1} c^{(2)}_{\bm{n};j} \frac{\Gamma(\alpha+j)}{\beta_2^{\alpha+j}} + w c_{\bm{n};0}^{(2)}.
\end{align}
\end{prop}
\begin{proof}
Apply Theorem~\ref{thm:type I thm} with  $\Psi(x) = x$, $M=1$, $\Phi_1(x) = x^2$, $\Phi_2(x) = 1$, $\bm{N} = (2,0)$. We choose
\begin{align}
    \widecheck\mu_1 & = \frac1x \mu_1, \\
    \widecheck\mu_2 & = \frac{1}{x} \mu_2 + w \delta_0.
\end{align}
Taking into account the multiplicity considerations from Section~\ref{general case} we obtain~\eqref{eq:LIdet1}, \eqref{eq:LIdet2} and~\eqref{eq:LID} 
with the last row consisting of the quantity
\begin{equation}
    B_{\bm{n}}(0)  = \widecheck\mu_1\Big[A^{(1)}_{\bm{n}}(x) \frac{x}{x}\Big] + \widecheck\mu_2\Big[A^{(2)}_{\bm{n}}(x) \frac{x}{x}\Big] = 
    \int_0^\infty \frac{A^{(1)}_{\bm{n}}(x)}{x} d\mu_1(x) + \int_0^\infty \frac{A^{(2)}_{\bm{n}}(x)}{x} d\mu_2(x)  + w A^{(2)}_{\bm{n}}(0).
\end{equation}
Note that unlike the type II case above, we {\it cannot} remove the term  $w A^{(2)}_{\bm{n}}(0)$ here (since $\Phi_2$ does not have a root at $x=0$). 

Finally, using~\eqref{eq:LagI}, ~\eqref{eq:LagI2} and
\begin{align}
    \int_0^\infty x^s d\mu_1 & = \int_0^\infty x^{\alpha+s} e^{-\beta_1 x} dx = \frac{\Gamma(\alpha+s+1)}{\beta_1^{\alpha+s+1}}, 
    \\
    \int_0^\infty x^s d\mu_2 & = \int_0^\infty x^{\alpha+s} e^{-\beta_2 x} dx = \frac{\Gamma(\alpha+s+1)}{\beta_2^{\alpha+s+1}}, 
\end{align}
for $s\ge -1$ (recall that $\alpha>0$), we obtain~\eqref{eq:exB}. 
\end{proof}
\begin{rem}
   Normality of the system $\widetilde{\bm\mu}$ in~\eqref{eq:LagPert1}, ~\eqref{eq:LagPert2} at a chosen location $\bm{n}$ depends on $w$. This normality fails if and only if $D_{\bm{n}}^{(I)} = 0$, where $D_{\bm{n}}^{(I)}$ is given in~\eqref{eq:LID}. It is easy to see that $D_{\bm{n}}^{(I)}$ is a polynomial of degree $\le 1$ in $w$, so there is at most one real value of $w$ for which $\widetilde{\bm\mu}$ is not normal at $\bm{n}$. Consequently, $\widetilde{\bm\mu}$  is always perfect if $w\in\bbC\setminus\bbR$, and there are at most countably many $w\in\bbR$ where it  is not perfect.
\end{rem}
\begin{rem}
   Applying the same argument to Proposition~\ref{prop:ex1}, we obtain that for each $\bm{n}$, $D_{\bm{n}}^{(I)}$ and $D_{\bm{n}}^{(II)}$  are non-zero multiples of each other, viewed as functions of $w$. It would be interesting to have a direct proof of this statement.  More generally, in the setting of Theorems~\ref{thm:type I thm} and~\ref{thm:type II thm}, it follows that $D_{\bm{n}}^{(I)} = 0$ if and only if $D_{\bm{n}}^{(II)} = 0$, but a direct argument for this is not known to us. Note that in the case $\Phi_1 = \dots = \Phi_r$ and $\Psi_1 = \dots = \Psi_r = 1$, this reduces to the Mahler identity from~\cite{Mahler}, see e.g., \cite{Geronimo}.
\end{rem}











\bibsection

\vspace{0.25cm}

\end{document}